\documentclass[11pt,reqno]{article}
\usepackage{hyperref}
\usepackage{nameref}
\usepackage[T1]{fontenc}
\usepackage{empheq}
\usepackage[titletoc]{appendix}
\usepackage[shortlabels]{enumitem}
\setlist[enumerate]{nosep}
\usepackage[doc]{optional}
\usepackage{xcolor}
\usepackage{float}
\usepackage{soul}
\usepackage{graphicx}
\definecolor{labelkey}{rgb}{0,0.08,0.45}
\definecolor{refkey}{rgb}{0,0.6,0.0}
\definecolor{Brown}{rgb}{0.45,0.0,0.05}
\definecolor{lime}{rgb}{0.00,0.8,0.0}
\definecolor{lblue}{rgb}{0.5,0.5,0.99}

 \usepackage{mathpazo}
\usepackage{nameref}

\usepackage{amsthm}
\usepackage{amsmath}

\usepackage{stmaryrd}
\usepackage{amssymb}

\newcommand{\aref}[1]{\hyperref[#1]{Appendix~\ref{#1}}}

\providecommand{\siff}{\Leftrightarrow}

\newcommand{\nnn}{\ensuremath{{n\in{\mathbb N}}}}

\newcommand{\menge}[2]{\big\{{#1}~\big |~{#2}\big\}}

\newcommand{\fenv}[1]%
{\ensuremath{\,\overrightarrow{\operatorname{env}}_{#1}}}
\newcommand{\benv}[1]%
{\ensuremath{\,\overleftarrow{\operatorname{env}}_{#1}}}

\newcommand{\RR}{\ensuremath{\mathbb R}}

\newcommand{\NN}{\ensuremath{\mathbb N}}

\providecommand{\BB}[2]{\operatorname{ball}(#1;#2)}

\newcommand{\dom}{\ensuremath{\operatorname{dom}}}
\newcommand{\argmin}{\ensuremath{\operatorname{argmin}}}
\newcommand{\prox}{\ensuremath{\operatorname{Prox}}}

\newcommand{\ran}{\ensuremath{\operatorname{ran}}}
\newcommand{\zer}{\ensuremath{\operatorname{zer}}}
\newcommand{\conv}{\ensuremath{\operatorname{conv}\,}}

\newcommand{\Id}{\ensuremath{\operatorname{Id}}}

{\begin{list}{}{%
\settowidth{\labelwidth}{\textrm{#1~}}%
\setlength{\leftmargin}{\labelwidth+\labelsep}}}
{\end{list}}

\usepackage{amsthm}
\usepackage[capitalize]{cleveref}
\crefname{equation}{}{equations}
\crefname{chapter}{Appendix}{chapters}
\crefname{item}{}{items}

\newtheorem{theorem}{Theorem}[section]

\newtheorem{lem}[theorem]{Lemma}

\newtheorem{cor}[theorem]{Corollary}
\newtheorem{proposition}[theorem]{Proposition}
\newtheorem{prop}[theorem]{Proposition}

\newtheorem{defn}[theorem]{Definition}
\newtheorem{thm}[theorem]{Theorem}
\newtheorem{example}[theorem]{Example}
\newtheorem{ex}[theorem]{Example}
\newtheorem{fact}[theorem]{Fact}

\newtheorem{rem}[theorem]{Remark}

\def\endproof{\ensuremath{\hfill \quad \blacksquare}

}

\providecommand{\ds}{\displaystyle}

\providecommand{\abs}[1]{\lvert#1\rvert}
\providecommand{\Abs}[1]{\Big\lvert#1\Big\rvert}
\providecommand{\norm}[1]{\lVert#1\rVert}
\providecommand{\normsq}[1]{\lVert#1\rVert^2}
\providecommand{\bk}[1]{\left(#1\right)}
\providecommand{\stb}[1]{\left\{#1\right\}}
\providecommand{\innp}[1]{\langle#1\rangle}

\providecommand{\RA}{\Rightarrow}

\providecommand{\grad}{\nabla}

\providecommand{\pat}{\partial}

\providecommand{\RR}{\mathbb{R}}

\providecommand{\conv}{\operatorname{conv}}
\providecommand{\ran}{\operatorname{ran}}

\providecommand{\intr}{\operatorname{int}}

\providecommand{\dom}{\operatorname{dom}}

\newcommand{\fix}{\ensuremath{\operatorname{Fix}}}

\providecommand{\parl}{\operatorname{par}}

\providecommand{\gra}{\operatorname{gra}}
\providecommand{\Id}{\operatorname{{ Id}}}

\providecommand{\fady}{\varnothing}

\providecommand{\argmin}{\mathrm{arg}\!\min}

\providecommand{\rras}{\rightrightarrows}

\providecommand{\NN}{\mathbb{N}}
\providecommand{\BB}[2]{\operatorname{ball}(#1;#2)}

\providecommand{\fix}{\operatorname{Fix}}
\providecommand{\ran}{\operatorname{ran}}
\providecommand{\rec}{\operatorname{rec}}
\providecommand{\Id}{\operatorname{Id}}

\providecommand{\pt}{{\partial}}

\providecommand{\zer}{\operatorname{zer}}
\providecommand{\inns}[2][w]{#2_{#1}}
\newcommand{\outs}[2][w]{{_#1}#2}

\providecommand{\DR}{\operatorname{DR}}
\providecommand{\FB}{\operatorname{FB}}

\providecommand{\TDR}[1][w,]{T_{{#1}{\DR}}}
\providecommand{\TFB}[1][w,]{T_{{#1}{\FB}}}

\providecommand{\fejer}{Fej\'{e}r }

\providecommand{\fady}{\varnothing}

\providecommand{\ri}{\operatorname{ri}}

\providecommand{\RR}{\mathbb{R}}
\providecommand{\NN}{\mathbb{N}}

\providecommand{\linop}{L}

\providecommand{\DR}{\operatorname{DR}}

\providecommand{\FB}{{\operatorname{FB}}}

\makeatletter
\def\namedlabel#1#2{\begingroup
   \def\@currentlabel{#2}%
   \label{#1}\endgroup
}
\makeatother

\definecolor{myblue}{rgb}{.8, .8, 1}
  \newcommand*\mybluebox[1]{%
    \colorbox{myblue}{\hspace{1em}#1\hspace{1em}}}

\allowdisplaybreaks 

\newcommand{\seppthree}{\setlength{\itemsep}{-3pt}}

\begin{document}

\title{\textsc
The forward-backward algorithm 
and the normal problem}

\author{
Walaa M.\ Moursi\thanks{
Mathematics, University of British Columbia, 
Kelowna, B.C.\ V1V~1V7, Canada
and 
Mansoura University, Faculty of Science, Mathematics Department, 
Mansoura 35516, Egypt. 
E-mail: \texttt{walaa.moursi@ubc.ca}.}}

\maketitle
\begin{abstract}
\noindent
The forward-backward splitting technique is a popular method
for solving monotone inclusions that has applications in optimization.
In this paper we explore the behaviour of the algorithm when 
the inclusion problem has no solution. We present a new formula to
define the normal solutions
using the forward-backward operator. We also provide a formula for the range of the displacement
map of the forward-backward operator.
Several examples illustrate our theory.
\end{abstract}
{\small

}

\section{Introduction}

Throughout this paper we work under the assumption that
\begin{empheq}[box=\mybluebox]{equation*}
\label{T:assmp}
X \text{~~is a real Hilbert space},
\end{empheq} 
with inner product $\innp{\cdot,\cdot}$ and
induced norm $\norm{\cdot}$. 
A (possibly) set-valued operator 
$A:X\rras X$ is \emph{monotone} if any two pairs 
$(x,u)$ and $(y,v)$ in the graph of $A$
satisfy
$\innp{x-y,u-v}\ge 0$, and 
 is \emph{maximally monotone} 
if it is monotone and 
any proper enlargement of the graph of 
$A$ (in terms of set inclusion) 
will no longer preserve the monotonicity of $A$.
In the following we assume that
\begin{empheq}[box=\mybluebox]{equation}
A \colon X\rras X \;\text{and}\;  B \colon X\rras X\;\;\text{are 
maximally monotone operators.}
\end{empheq}
Thanks to the fact that the \emph{subdifferential} 
operator associated with a convex lower semicontinuous 
proper function is a maximally monotone operator 
(see \cref{fact:func:sd} below), 
the notion of monotone operators becomes 
of significant importance in optimization
and nonlinear analysis.
For further discussion on monotone operator theory and
its connection to optimization see, e.g., the books
\cite{BC2011}, 
\cite{BorVanBook},
\cite{Brezis}, 
\cite{BurIus},
\cite{Simons1},
\cite{Simons2},
\cite{Zeidler2a},
\cite{Zeidler2b},
and
\cite{Zeidler1}. 

The problem of finding a zero of the sum of two maximally 
monotone operators
$A$ and $B$
is to find $x\in X$ such that $x\in (A+B)^{-1}0$.
When specializing $A$ and $B$  to subdifferential
operators of convex lower semicontinuous proper 
functions,
the problem is equivalent to 
finding a minimizer of the sum of the two functions, 
which is a classical optimization problem.

Suppose that $A$ is \emph{firmly nonexpansive}\footnote{We 
point out that the assumption of that 
$A$ is firmly nonexpansive
can be relaxed to $A$ is cocoercive 
(see \cref{rem:fnon:coco}).}(see \cref{nexp:f:nexp}).
Let $x_0\in X$ and let $\TFB[]$ be
the forward-backward operator associated 
with the pair $(A,B)$ (see \cref{fb:duality}). 
When $(A+B)^{-1}0\neq \fady $
the 
sequence
$(\TFB[]^n x_0)_{n\in \NN}$
produced by iterating the forward-backward operator
 converges 
weakly\footnote{For general conditions on \emph{strong}
convergence of the forward-backward algorithm 
we refer the reader to \cite{ABC10}.} to a point
in $(A+B)^{-1}0=\fix \TFB[]
=\menge{x\in X}{x=\TFB[]x}$
(see, e.g., \cite{Tseng91}, \cite{Lemaire97} or \cite{Comb04}). 
Applications of this setting appear in convex optimization  
(see, e.g., \cite[Section~27.3]{BC2011}), 
evolution inclusions (see, e.g., \cite{ABC10})
and inverse problems (see, e.g., \cite{CombWaj05} 
and \cite{CombDVu10}).


\emph{The goal of this work is to examine 
the forward-backward
operator in the inconsistent case,  
i.e., when $(A+B)^{-1}0=\fady$, 
using the framework of the normal problem
introduced in \cite{Sicon2014}. In this case $\fix \TFB[]=\fady$,
and the classical analysis, which uses the advantage
 of iterating an \emph{averaged}
operator (see \cref{nexp:f:nexp} below)
that has a fixed point, is no longer applicable.}

Let us summarize the main contributions of the 
paper:
\begin{itemize}
\item[{\bf R1}]\namedlabel{R:1}{\bf R1}
We provide a systematic study of the 
forward-backward operator
when the sum problem is possibly inconsistent.
This is mainly illustrated in \cref{prop:eqv}
where we establish the connection between the 
perturbed problem
introduced in \cite{Sicon2014} 
and the forward-backward operator.

\item[{\bf R2}]\namedlabel{R:2}{\bf R2}
We prove that the range of the displacement operator
associated with the forward-backward operator $\TFB[]$
coincides with that of the Douglas-Rachford operator $\TDR[]$.
Consequently, the minimal displacement vectors 
associated with  
$\TFB[]$ and $\TDR[]$ coincide (see \cref{cor:conc:old}). 
This gives an alternative approach to define the normal
problem introduced in \cite{Sicon2014}.

\item[{\bf R3}]\namedlabel{R:3}{\bf R3}
A significant consequence of 
\ref{R:2} is that it allows to 
use the advantage of the self-duality of $\TDR[]$
(which does not hold for $\TFB[]$ 
as we illustrate in \cref{ex:TFB:not:sd}) 
to draw more conclusions
about $\TFB[]$.
In particular, in \cref{cor:conc} we provide a formula for the
range of the displacement operator in terms of the ranges of the 
underlying operators using the notion of \emph{near equality}.
The result simplifies to more elegant formulae when specializing 
the 
operators to subdifferential operators as illustrated  
in \cref{prop:func:fd}.
Our results are sharp in the sense that 
near equality cannot be replaced by equality 
which we illustrate in \cref{ex:neq:not:eq}.
\item[{\bf R4}]\namedlabel{R:4}{\bf R4}
In the case when $A$ and $B$ are affine, we
prove that, in the consistent case, the sequence 
produced by iterating $\TFB[]$ converges 
\emph{strongly} to the \emph{nearest} point 
in the set of zeros of the sum.
If $X$ is finite-dimensional, we also get \emph{linear} rate
of convergence (see \cref{prop:FB:app}).
\end{itemize}

The remainder of this paper is organized as
follows: \cref{nexp:f:nexp} provides facts and auxiliary results
concerning averaged and (firmly) nonexpansive operators.
In \cref{fb:duality}, we provide an overview of the Attouch-Th\'{e}ra
duality and formulate the primal and dual solutions using the 
forward-backward operator. Our main results start in
\cref{s:fb:normal}, which deals with the normal problem and 
the connection to the forward-backward operator.
In \cref{s:fb:ran}, we explore the range of the 
displacement operator associated with 
the forward-backward operator.
In \cref{s:aff:app}, we study the asymptotic behaviour
of \emph{asymptotically regular} affine nonexpansive operators
in the possibly fixed point free setting.
An application to the forward-backward algorithm is provided as well.
Finally in \cref{s:alg:consq} we provide some algorithmic consequences.

\subsection*{Notation}
Let $C$ be a nonempty closed convex 
subset of $X$. 
We use $\iota_C$, $N_C$
and $P_C$ to denote
the \emph{indicator function},
the \emph{normal cone}
operator and the \emph{projector (this is also known as 
nearest point mapping)}
associated with $C$, respectively.
Let $f\colon X\to \left]-\infty,+\infty\right]$ be convex, lower semicontinuous, and proper.
The \emph{subdifferential} of $f$ is the (possibly)
set-valued operator
$\pt f\colon X\rras X:x\to\menge{u\in X}{(\forall y\in X)~
f(y)\ge f(x)+\innp{u,y-x}}$.
Let $\Id:X\to X$ be the identity operator. 
The \emph{resolvent} of $A$ is $J_A:=(\Id+A)^{-1}$
and the \emph{reflected resolvent} is
$R_A:=2J_A-\Id$. 
Otherwise, the notation we adopt is standard and follows, 
e.g., \cite{BC2011} and \cite{Rock98}.
\section{Averaged and (firmly) nonexpansive operators}
\label{nexp:f:nexp}
Let $T:X\to X$. Then $T$ is \emph{nonexpansive}
if 
\begin{equation}
(\forall x\in X)
(\forall y\in X)
\quad\norm{Tx-Ty}\le \norm{x-y};
\end{equation}
$T$ is \emph{firmly nonexpansive}
if
\begin{equation}
(\forall x\in X)
(\forall y\in X)
\quad\normsq{Tx-Ty}+\normsq{(\Id-T)x-(\Id-T)y}\le \normsq{x-y};
\end{equation}
and $T$ is \emph{averaged} if 
there exists 
$\alpha \in \left]0,1\right[$
and a nonexpansive operator $N:X\to X$ such that 
\begin{equation}
T=(1-\alpha )\Id+\alpha N.
\end{equation}
\begin{fact}
\label{JA:fne:orig}
The following hold:
\begin{enumerate}
\item
\label{JA:fne:orig:i}
$J_A$ is single-valued, maximally monotone and 
firmly nonexpansive.
\item 
\label{JA:fne:orig:ii}
(The inverse resolvent identity)
$J_{A^{-1}}=\Id-J_A$.
\end{enumerate}
\end{fact}
\begin{proof}
\ref{JA:fne:orig:i}:
See 
\cite[Corollary~on~page~344]{Minty2}
 and \cite[Proposition~1(c)]{Rock76}.
 \ref{JA:fne:orig:ii}:
 See, e.g., \cite[Lemma~12.14]{Rock98}.
\end{proof}
In the sequel we make use of 
the useful characterization  
(see, e.g., 
\cite[Equation~11.1~on~page~42]{GReich84}):
\begin{equation}
\label{def:coco}
T\text{~is firmly nonexpansive~}\siff
(\forall x\in X)
(\forall y\in X)
\quad\normsq{Tx-Ty}
\le 
\innp{x-y, Tx-Ty}.
\end{equation}

\begin{defn}[{\bf asymptotic regularity of 
operators vs. sequences}]
Let $T:X\to X$ and let $(x_n)_\nnn$
be a sequence in $X$.
Then $T$ is asymptotically regular if $(\forall x\in X)$
$T^n x-T^{n+1}x\to 0$ and $(x_n)_\nnn$ 
is asymptotically regular if $x_n-x_{n+1}\to 0$.
\end{defn}
\begin{fact}
\label{f:av:asreg}
Suppose that $T:X\to X$ is averaged; 
in particular, firmly nonexpansive.
Then $T$ is asymptotically regular.
\end{fact}
\begin{proof}
See \cite[Corollary~1.1~\&~Proposition~2.1]{Br-Reich77}
or \cite[Proposition~5.15(ii)~\&~Corollary~5.16(ii)]{BC2011}.
\end{proof}

\begin{fact}
\label{fact:inf:dis:vec}
Suppose that $T:X\to X$ is nonexpansive.
Then $\overline{\ran}(\Id -T)$ is nonempty 
closed and convex.
Consequently
the \emph{minimal displacement vector} 
 associated  with $T$ 
 is the unique well-defined vector
 \begin{equation}
 \label{def:vt}
 v_{T}:=P_{\overline{\ran} (\Id-T)} 0.
 \end{equation}
\end{fact}
\begin{proof}
See \cite{Ba-Br-Reich78},
\cite{Br-Reich77} or \cite{Pazy}.
\end{proof}
Unless otherwise stated, throughout this paper
we assume that
\begin{empheq}[box=\mybluebox]{equation*}
T:X\to X\;\;\text{
is nonexpansive.}
\end{empheq}

The following result is well-known 
when $T$ is firmly nonexpansive.
We include a simple proof, when $T$ is \emph{averaged}, 
for the sake 
of completeness (see also \cite[Lemma~3.9]{BDM:LNA:15}).
\begin{prop}
\label{prop:av:tele}
Suppose that $T$ is 
averaged
 and that $v_T:=P_{\overline{\ran}(\Id -T)}0\in
 \ran(\Id-T)$. Let $x\in X$.
 Then the following hold:
 \begin{enumerate}
\item
\label{prop:av:tele:i}
$\sum_{n=0}^\infty\normsq{T^n x-T^{n+1}x-v_T}<+\infty$.
 \item
 \label{prop:av:tele:ii}
$
T^nx-T^{n+1}x\to v_T,
$
equivalently; the sequence 
$(T^n x+nv_T)_\nnn$ is \emph{asymptotically regular}.
\end{enumerate}
\end{prop}
 
\begin{proof}
It follows from
\cite[Lemma~2.1]{Comb04}
that
$(\exists \alpha \in \left]0,1\right[)$
such that
$(\forall x\in X)$ $(\forall y\in X)$
\begin{equation}
\label{eq:av:ineq}
\normsq{(\Id-T)x-(\Id-T)y}\le 
\frac{\alpha}{1-\alpha}\bk{\normsq{x-y}-\normsq{Tx-Ty}}.
\end{equation}
Moreover \cite[Proposition~2.5(vi)]{BM2015:AFF}
implies that $(T^n x+nv_T)_\nnn$ is \fejer monotone
with respect to $\fix (v_T+T)$.
Now let $n\in \NN$ and let
$y_0\in \fix (v_T+T)$.
Using \cite[Proposition~2.5(iv)]{{BM2015:AFF}}
we learn that $T^n y_0=y_0-nv_T$.
It follows from \cref{eq:av:ineq} applied with
$(x,y)$ replaced by $(T^n x, T^n y_0)$
that
\begin{subequations}
\label{eq:av:tele}
\begin{align}
\normsq{T^nx-T^{n+1}x- v_T}
&=\normsq{(\Id-T)T^n x-(\Id-T)T^ny_0}\\
&\le \frac{\alpha}{1-\alpha}\bk{\normsq{T^n x-T^n y_0}
- \normsq{T^{n+1} x-T^{n+1} y_0}}.
\end{align}
\end{subequations}
\ref{prop:av:tele:i}:
This follows from \cref{eq:av:tele}
by telescoping.
\ref{prop:av:tele:ii}:
This is a direct consequence of \ref{prop:av:tele:i}.
\end{proof}
\begin{proposition}
\label{prop:int:notfady}
Suppose that $v_T:=P_{\overline{\ran}(\Id -T)}0\in
 \ran(\Id-T)$ and 
that $\intr \fix(v_T+T)\neq \fady$. Then the following hold:
\begin{enumerate}
\item
\label{prop:int:notfady:i}
$\sum_{n=0}^\infty\norm{T^n x-T^{n+1}x-v_T}<+\infty$.
\item
\label{prop:int:notfady:ii}
 $(T^n x+nv_T)_\nnn$ converges strongly.
 \end{enumerate}
\end{proposition}
\begin{proof}
The proof follows along the lines
of \cite[Proposition~5.10]{BC2011}.
\ref{prop:int:notfady:i}:
Let $x\in \fix (v_T+T)$ and let $r>0$ such that
$\BB{x}{r}\subseteq \fix (v_T+T)$. 
Obtain a sequence $(y_n)_\nnn$
defined as:
\begin{equation}
\label{eq:cases:int}
(\forall \nnn)\quad
y_n
=\begin{cases}
x,&\text{if~} x_{n+1}=x_n;\\
x-r\frac{x_{n+1}-x_n}{\norm{x_{n+1}-x_n}},
&\text{otherwise}.
\end{cases}
\end{equation}
Then $(y_n)_\nnn\subseteq \BB{x}{r}$.
Set $(\forall \nnn)$ $x_n:=T^n x+nv_T$.
It follows from \cite[Proposition~2.5(vi)]{BM2015:AFF}
that the sequence $(x_n)_\nnn$
is \fejer monotone with respect to $\fix (v+T)$, therefore
$(\forall \nnn) $  
$\normsq{x_{n+1}-y_n}\le \normsq{x_n-y_n}$;
equivalently
$(\forall \nnn) $  
$\normsq{x_{n+1}-x+(x-y_n)}\le \normsq{x_n-x+(x-y_n)}$.
Expanding and simplifying in view of
\cref{eq:cases:int} yield
$(\forall \nnn) $  
$\normsq{x_{n+1}-x}\le \normsq{x_n-x}-2\innp{x_n-x_{n+1},x-y_n}=
 \normsq{x_n-x}-2r\norm{x_n-x_{n+1}}
 $.
Telescoping yields
\begin{equation}
\label{eq:cauchy}
\sum_{n=0}^\infty\norm{x_n-x_{n+1}}\le 
\frac{1}{2r} \normsq{x_0-x}.
\end{equation}
  \ref{prop:int:notfady:ii}:
It follows from \cref{eq:cauchy} that
 $(x_n)_\nnn=(T^n x+nv)_\nnn$ is a Cauchy sequence
   and therefore it converges. 
\end{proof}
Let $S$ be nonempty subset of
$X$ and let $a\in X$. Before we proceed further we 
need the 
following useful translation formula (see, e.g., 
\cite[Proposition~3.17]{BC2011}).
\begin{equation}
\label{eq:trans:form}
(\forall x\in X) \quad
P_{a+S}x=a+P_S(x-a).
\end{equation}
\begin{ex}
Let $n\ge 1$.
Suppose\footnote{Let $\nnn$.
The \emph{positive orthant} in 
$\RR^n$ is $\RR^n_{+}=\left[0,+\infty\right[^n$
and the \emph{strictly positive orthant} 
in $\RR^n$ is $\RR^n_{++}=\left]0,+\infty\right[^n$.
Likewise we define the \emph{negative orthant} and the
\emph{strictly negative orthant} 
$\RR^n_{-}$ and $\RR^n_{--}$, respectively.
} 
that $X=\RR^n$, that $p\in \RR^n_{++}$
and that $T=p+P_{ \RR^n_{+}}$.
Then $T$ is (firmly) nonexpansive,
$\fix T=\fady$,
$\ran(\Id-T)=-p+ \RR^n_{-}$,
$v_T=-p\in
 \ran(\Id-T)$
 and
$\intr \fix(v_T+T)= \RR^n_{--}\neq \fady$.
 Consequently 
$\sum_{n=0}^\infty\Abs{T^n x-T^{n+1}x-v_T}<+\infty$
and $(T^n x+nv_T)_\nnn$ converges.
\end{ex}
\begin{proof}
The claim that $T$ is firmly nonexpansive (hence nonexpansive)
follows from e.g., \cite[Section~3]{GReich84}.
Now $\Id-T=\Id-p-P_{\RR^n_{+}}=-p+ P_{\RR^n_{-}}$,
hence $\ran (\Id-T)=-p+ \RR^n_{-}$
and $\fix T= \fady\siff 0\not\in \ran (\Id-T)=-p+ \RR^n_{-}\siff p\not\in \RR^n_{-}$, which is true.
Using \cref{eq:trans:form} with $(a, S)$ replaced by
$(-p, \RR^n_{-})$ we have 
$v_T=P_{-p+ \RR^n_{-}}0=-p+P_{\RR^n_{-}}p=-p$.
Consequently $v_T+T=-p+p+P_{ \RR^n_{+}}=P_{ \RR^n_{+}}$ 
and therefore $\fix (v_T+T)= \RR^n_{+}$
which implies that $\intr \fix (v_T+T)= \RR^n_{++}$.
Now apply \cref{prop:int:notfady}.
\end{proof}

\begin{cor}
\label{cor:RR:conv}
Suppose that $X=\RR$, that 
$\fix T=\fady$ and that
$v_T:=P_{\overline{\ran}(\Id -T)}0\in
 \ran(\Id-T)$.
Then $\intr \fix(v_T+T)\neq \fady$ and 
$\sum_{n=0}^\infty\Abs{T^n x-T^{n+1}x-v_T}<+\infty$.
Consequently $(T^n x+nv_T)_\nnn$ converges.
\end{cor}
\begin{proof}
It follows from \cite[Proposition~2.5(i)]{BM2015:AFF}
that $\fix (v+T)$ contains an unbounded interval,
and therefore, since $X=\RR$, we conclude that
$\intr \fix (v+T)\neq \fady$.
Now apply \cref{prop:int:notfady}.
(See also \cite[Theorem~3.6]{BDM:LNA:15}).
\end{proof}

%

\section{The forward-backward operator and duality}
\label{fb:duality}
The \emph{primal} problem for
the ordered pair
$(A,B)$ is 
\begin{equation}
\label{P:sum}
{\rm(P)} \;\; \text{find}\;\; x\in X \;\;\text{such that}\;\;
0\in Ax+Bx.
\end{equation}
The Attouch-Th\'{e}ra
\emph{dual} pair \cite{AT} for the ordered pair $(A,B)$
is the pair\footnote{Let $B:X\rras X$. Then 
$B^{\ovee}:=(-\Id)\circ B\circ (-\Id)$ and 
$B^{-\ovee}:=(B^{-1})^\ovee=(B^\ovee)^{-1}$ 
(see \cite[Equation~(10)]{JAT2012}).} 
$(A^{-1},B^{-\ovee})$
and the corresponding \emph{dual problem} is
\begin{equation}
\label{D:sum}
{\rm(D)} \;\; \text{find}\;\; x\in X \;\;\text{such that}\;\;
0\in A^{-1}x+B^{-\ovee}x.
\end{equation}
The sets of primal and dual
solutions for the ordered pair $(A,B)$, denoted respectively
by $Z$ and $K$ are
\begin{equation}
Z:=(A+B)^{-1}(0) \quad \text{and} 
\quad K:=(A^{-1}+B^{-\ovee})(0).
\end{equation}

From now on we assume that
\begin{empheq}[box=\mybluebox]{equation}
A:X\to X \;\;\text{is 
firmly nonexpansive.   }\;\;
\end{empheq}

The \emph{forward-backward} algorithm to solve 
\cref{P:sum} iterates the 
operator
\begin{equation}
\label{eq:def:FB}
T_{\FB}:=T_{\FB(A,B)}:=J_B(\Id-A).
\end{equation}
On the other hand the \emph{Douglas-Rachford}
algorithm to solve 
\cref{P:sum} iterates the operator
\begin{equation}
\label{eq:def:DR}
T_{\DR}:=T_{\DR(A,B)}:=\Id-J_A+J_BR_A.
\end{equation}
Let $x\in X$. If $Z\neq \fady $ then
each of the sequences
$(T_{\FB}^n x)_\nnn$ 
(see, e.g., \cite[Corollary~6.5]{Comb04} 
or \cite[Section~25.3]{BC2011})
 and $(J_A T_{\DR}^n x)_\nnn$ 
 (see, e.g., \cite{Svaiter} or \cite{L-M79})
 converges weakly to 
 a (possibly different) solution of
 \cref{P:sum}.
 \begin{rem}
\label{rem:fnon:coco}
Let $\alpha>0$.
Since $\zer (A+B)=\zer (\alpha A+\alpha B)$,
the assumption that $A$
is firmly nonexpansive could be replaced by
$A$ is $\alpha$-cocoercive\footnote{Recall that $A:X\to X$
is cocoercive
if $(\exists \alpha>0)$ such that $\alpha A$
is firmly nonexpansive.}. 
In this case 
\cref{eq:def:FB} and \cref{eq:def:DR}
can be applied with the ordered pair
$(A,B)$ is replaced 
by $(\alpha A,\alpha B)$. 
\end{rem}

 \begin{defn}[{\bf paramonotone and $3^*$
 monotone operators}]  
Let $C\colon X\rras X$ be monotone.
Then 
\begin{enumerate}
\item
$C$ is \emph{paramonotone}\footnote{For 
detailed discussion and examples of
paramonotone operators we refer the reader to \cite{Iusem98}.}
 if 
$(\forall (x,u)\in \gra C)$
$(\forall (y,v)\in \gra C)$ we have
\begin{equation}
\left.
\begin{array}{c}
 (x,u)\in \gra C\\
(y,v)\in \gra C\\
\innp{x-y,u-v}=0
\end{array}
\right\}
\quad\RA\quad 
\big\{(x,v),(y,u)\big\}\subseteq \gra C.
\end{equation}
\item
$C$ is \emph{$3^*$ monotone}\footnote{For 
detailed discussion and examples of
$3^*$ monotone operators we refer the
reader to \cite{Br-H}.} 
(this is also known as \emph{rectangular}) if 
\begin{equation}
(\forall x\in \dom C)(\forall v\in \ran C)
\qquad \inf_{(z,w)\in \gra C}\innp{x-z,v-w}>-\infty.
\end{equation}
\end{enumerate}
\end{defn}
\begin{lem}
\label{lem:coc:prop}
The following hold:
\begin{enumerate}
\item
\label{lem:coc:mm}
$A$ is maximally monotone.
\item
\label{lem:coc:para}
$A$ is paramonotone.
\item
\label{lem:coc:3*}
$A$ is $3^*$ monotone.
\end{enumerate}
\end{lem}
\begin{proof}
\ref{lem:coc:mm}:
This is \cite[Example~20.27]{BC2011}.
\ref{lem:coc:para} \&
\ref{lem:coc:3*}: Note that $A=\Id-(\Id-A)$
 and $\Id -A$ is firmly nonexpansive.
 The conclusion follows from 
 \cite[Theorem~6.1]{BWY2014}.
\end{proof}

\begin{prop}
\label{FB:collec}
The following hold:
\begin{enumerate}
\item
\label{FB:collec:i}
$\TFB[]$ is averaged.
\item
\label{FB:collec:i:i}
$\TFB[]$ is asymptotically regular.
 \item
 \label{FB:collec:iv}
 $K$ is a singleton.
 \item
\label{FB:collec:ii}
$Z=\fix \TFB[]$.
\item
\label{FB:collec:iii}
$K=A(Z)=A(\fix \TFB[])$.

 \end{enumerate}
\end{prop}
\begin{proof}
\ref{FB:collec:i}:
Since $A$ is firmly nonexpansive
so is $\Id-A$ (see, e.g., \cite[Lemma~2.3]{Comb04}). 
Note that $J_B$  is firmly nonexpansive by
\cite[Proposition~1(c)]{Rock76}.
It follows from \cite[Remark~4.24(iii)]{BC2011}
that $\Id-A$ and $J_B$ are $1/2$-averaged
and therefore $T=J_B(\Id -A)$ is $2/3$-averaged
by 
\cite[Lemma~2.2(iii)]{Comb04}.
\ref{FB:collec:i:i}:
Combine \ref{FB:collec:i} and \cref{f:av:asreg}.
\ref{FB:collec:iv}: 
Let $k_1$ and $k_2$ be in $K$.
It follows from \cite[Proposition~2.4]{JAT2012}
that $(\exists z_i \in Z)$ 
such that $k_i\in Az_i\cap(-Bz_i)=Az_i$,
$i\in \stb{1,2}$.
 Since $A$ is single-valued, we conclude that
 $k_i=Az_i$, $i\in \stb{1,2}$.
 Using \cite[Corollary~2.13]{JAT2012}
 we learn that $\innp{z_1-z_2,k_1-k_2}
 =\innp{z_1-z_2, Az_1-Az_2}=0$.
 Now combine with \cref{lem:coc:prop}\ref{lem:coc:para}
  and use that $A$ is single-valued to learn that $k_1=k_2$. 
\ref{FB:collec:ii}:
This follows from \cite[Proposition~25.1(iv)]{BC2011}.
\ref{FB:collec:iii}: 
In view of \ref{FB:collec:iv}, let 
$K=\stb{k}$.
It follows from \cite{JAT2012}
that $(\forall z\in Z)$ $k=Az\cap(-Bz)$,
which implies, since $A$ is single-valued, 
that $k=Az$; equivalently $K=\stb{k}=A(Z)$.
Now combine with \ref{FB:collec:ii}.
\end{proof}
\begin{fact}[{\bf Baillon-Haddad}]
\label{f:grad:fne}
Let $f\colon X\to \RR$ be convex and differentiable. 
Then 
\begin{equation}
\grad f \;\;\text{is nonexpansive}\; \siff \;\grad f\;\;
\text{is firmly nonexpansive}.
\end{equation}
\end{fact}
\begin{proof}
See \cite[Corollaire~10]{BH1979}.
\end{proof}

\begin{fact}
\label{fact:func:sd}
Let $f\colon X\to \left]-\infty,+\infty\right]$ 
be convex, lower semicontinuous, and proper.
Then the following hold:
\begin{enumerate}
\item
\label{fact:func:sd:mm}
$\pt f$ is maximally monotone.
\item
\label{eq:sb:inv:conj} 
$(\pt f)^{-1}=\pt f^*.$
\end{enumerate}
\end{fact}
\begin{proof}
\ref{fact:func:sd:mm}: See, e.g., \cite[Theorem~A]{Rock1970}. 
\ref{eq:sb:inv:conj}: See, e.g., \cite[Remark~on~page~216]{Rock1970},
\cite[Th\'{e}or\`{e}me~3.1]{Gossez}, 
or \cite[Corollary~16.24]{BC2011}.
\end{proof}
Suppose that $C$ is a nonempty closed convex subset of $X$.
It is well-known (see, e.g., \cite[Example~23.4]{BC2011}) that 
\begin{equation}
\label{eq:res:proj}
J_{N_C}=P_C.
\end{equation}
\begin{prop}
\label{prop:func}
Suppose that 
$f\colon X\to \RR$ is convex and differentiable 
such that $\grad f$ is nonexpansive
and that
$g\colon X\to \left]-\infty,+\infty\right]$
is convex, lower semicontinuous, and proper.
Suppose that $A=\grad f$ and that
$B=\pat g$. 
Then the following hold\footnote{Let $h:X\to \left]-\infty,+\infty\right]$
be proper. The \emph{set of minimizers of} $h$,
$\stb{x\in X~|~h(x)=\inf h(X)}$, is 
denoted by $\argmin h$.}$^{,}$\footnote{Suppose that 
$g\colon X\to \left]-\infty,+\infty\right]$
is convex, lower semicontinuous, and proper.
Then $\prox_g$ is the \emph{Moreau prox operator}
 associated with $g$ defined 
by $\prox_g:X\to X:x\mapsto (\Id+\pt g)^{-1}(x)
=\argmin_{y\in X} 
\bk{g(y)+\tfrac{1}{2}\norm{x-y}^2}$.
}:
\begin{enumerate}
\item
\label{prop:func:i}
$\fix \TFB[]=\zer(\grad f+\pt g)=\argmin(f+g)$.
\item
\label{prop:func:ii}
$\TFB[]=\prox_g(\Id-\grad f)$.
 \end{enumerate}   
 If in addition,
  $g=\iota_V$ where $V$
 is a nonempty closed convex subset of $X$,
 then we have
 \begin{enumerate}
 \setcounter{enumi}{2}
 \item  
 \label{prop:func:iv}
 $\TFB[]=P_V(\Id-\grad f)$.
 \end{enumerate}
\end{prop}
\begin{proof}
Note that $\dom f=X$
 and that $A=\grad f$ is
 firmly nonexpansive by \cref{f:grad:fne}.
\ref{prop:func:i}:
The first identity is \cref{FB:collec}\ref{FB:collec:ii}
applied with $(A,B)$ replaced by $(\grad f,\pt g)$.
It follows from 
\cite[Proposition~3.2~\&~Corollary~3.4]{Burachik}
that $A+B=\grad f+\pt g=\pt (f+g)$.
Now apply \cite[Proposition~26.1]{BC2011}.
\ref{prop:func:ii}:
Combine \cref{eq:def:FB} and \cite[Example~23.3]{BC2011}.
 \ref{prop:func:iv}:
 Combine \ref{prop:func:ii} and 
 \cref{eq:res:proj}.
\end{proof}
\begin{rem}
Let 
$f\colon X\to \RR$ be convex and differentiable
with 
$1/\beta$ Lipschitz continuous gradient,
where $\beta>0$. 
Then $\beta \grad f$ is nonexpansive, hence 
firmly nonexpansive by \cref{f:grad:fne}.
Since $\argmin(f+g)=\argmin(\beta f+\beta g)$,
 \cref{prop:func} 
can be applied, with $(f,g)$ replaced by $(\beta  f, \beta g)$,
to find a minimizer of $f+g$.
\end{rem}
Suppose that\footnote{Let $C$ be a 
nonempty closed convex
subset of $X$. We use $d_C$ to denote the \emph{distance}
from the set $C$ defined by $d_C:X\to \left[0,+\infty\right[:
x\mapsto \min_{c\in C}\norm{x-c}=\norm{x-P_C x}$. 
} 
  $C$ is a nonempty closed convex
subset of $X$.
In the sequel we make 
use of the following useful result (see, e.g., \cite[Exemple~on~page~286]{Moreau65}  
or
\cite[Corollary~12.30]{BC2011}).
\begin{equation}
\label{eq:grad:dist2}
\grad \bk{\tfrac{1}{2}d^2_C}=\Id-P_C.
\end{equation}
\begin{example}[{\bf Method of Alternating Projections (MAP) as a forward-backward iteration}]
\label{lem:FB:MAP}
Suppose that $U$ and $V$ are nonempty closed convex subsets of 
$X$, that $f=\tfrac{1}{2}d^2_U$ and that $g=\iota_V$.
Suppose that $A=\grad f=\Id-P_U$ and that $B=\pt g=N_V$.
Then $A$ is firmly nonexpansive and 
\begin{equation}
T_{\FB(\Id-P_U,N_V)}=P_VP_U.
\end{equation}
\end{example}
\begin{proof}
It follows from 
\cref{eq:grad:dist2}
that 
$\grad f=\Id-P_U$,
which 
 is firmly nonexpansive by e.g.,
\cite[Equation~1.7~on~page~241]{Zar71:1}.
Moreover 
\cref{eq:res:proj}
implies
that $J_B=J_{N_V}=P_V$. Consequently,
$
T_{\FB(A,B)}=J_B\circ(\Id-A)=P_V(\Id-(\Id-P_U))=P_VP_U.
$
\end{proof}

\section{The forward-backward operator and the normal problem}
\label{s:fb:normal}
Let $C:X\rras X$ and let $w\in X$. 
The \emph{inner shift}
and \emph{outer shift} of an operator $C$
 by $w$  at $x\in X$ are defined by
\begin{empheq}[box=\mybluebox]{equation}
C_{w}x:=C(x-w)\;\;\text{ and}\;\;  _{w}Cx:=-w+Cx,
\end{empheq}
respectively.

Let $w\in X$. 
The \emph{$w$-perturbed problem}
introduced in \cite{Sicon2014} is:
\begin{equation}
\label{P:DR}
(P_{w}) \;\; \text{find}\;\; x\in X \;\;\text{such that}\;\;
0\in \outs[w]{A}x+\inns[w]{B}x=Ax+B(x-w)-w,
\end{equation}
 and the corresponding set of zeros is
\begin{equation}
\label{eq:def:Zw}
Z_w:=\menge{x\in X}{0\in 0\in \outs[w]{A}x+\inns[w]{B}x}
=\menge{x\in X}{w\in Ax+B(x-w)}.
\end{equation}

\begin{prop}
\label{prop:eqv}
Let $w\in X$. Then 
\begin{equation}
\label{eq:TFB:shift}
{\TFB[]}_{(\outs[w]{A},\inns[w]{B})}={_{-w}T_{\FB}}=w+\TFB[],
\end{equation}
and 
\begin{equation}
\label{eq:Zw:fix}
Z_w=\fix (w+\TFB[]).
\end{equation}
Moreover, the following are equivalent:
\begin{enumerate}
\item
\label{prop:eqv:i}
$Z_w\neq \fady$.
\item
\label{prop:eqv:ii}
$  w\in\ran(A+\inns[w]{B})$.
\item
\label{prop:eqv:iv}
$w\in \ran (\Id-\TFB[]) $.
\item
\label{prop:eqv:iii}
$w\in \ran (\Id-\TDR[]) $.
\end{enumerate}
\end{prop}

\begin{proof}
Let $x\in X$. Using \cref{eq:def:FB} and 
\cite[Proposition~23.15(ii)\&(iii)]{BC2011}
we have
$\TFB[](\outs[w]{A},\inns[w]{B})x=J_{\inns[w]{B}}(\Id-\outs[w]{A})x=
J_B((x-(Ax-w))-w)+w=J_B(x-Ax)+w=J_B(\Id-A)x+w=w+\TFB[]x,
$ which proves \cref{eq:TFB:shift}.
To prove \cref{eq:Zw:fix}
apply \cref{FB:collec}\ref{FB:collec:ii}
with $(A,B)$ replaced by $ (\outs[w]{A},\inns[w]{B})$
 and use \cref{eq:TFB:shift}.
``\ref{prop:eqv:i}$\Leftrightarrow$\ref{prop:eqv:ii}":
This follows from \cref{eq:def:Zw}.
``\ref{prop:eqv:i}$\Leftrightarrow$\ref{prop:eqv:iv}":
 Indeed, using \cref{eq:Zw:fix}
 we have $Z_w\neq \fady\siff \fix (w+\TFB[])
 \neq \fady \siff (\exists x\in X)
 $ such that $x=w+\TFB[]x\siff w\in \ran(\Id-\TFB[])$.
``\ref{prop:eqv:i}$\Leftrightarrow$\ref{prop:eqv:iii}":
This follows from  \cite[Proposition~3.3]{Sicon2014}.
 \end{proof}
 \begin{thm}
\label{cor:conc:old}
We have\footnote{For convenience we shall use 
$v_{\FB}$ and $v_{\DR}$ to denote $v_{\TFB[]}$ and
$v_{\TDR[]}$ respectively.}
\begin{enumerate}
\item
\label{cor:conc:i}
$\ran (\Id-\TDR[])=\ran (\Id-\TFB[])$.
\item
\label{cor:conc:i:i}
$\ran (\Id-\TFB[])\subseteq \ran A+\ran B$.
\item 
\label{cor:conc:iii}
$v_{\DR}=v_{\FB}$.
\end{enumerate}
\end{thm}
\begin{proof}
\ref{cor:conc:i}:
This is clear from the equivalence of
\ref{prop:eqv:iv} and \ref{prop:eqv:iii}
in \cref{prop:eqv}.
\ref{cor:conc:i:i}:
Combine \ref{cor:conc:i}
 and \cite[Proposition~4.1]{EckThesis}.
\ref{cor:conc:iii}:
Indeed, using \ref{cor:conc:i} 
 and \cref{def:vt} we have
$v_{\DR}
=P_{\overline{\ran}(\Id-T_{\DR})}0
=P_{\overline{\ran}(\Id-\TFB[])}0
=v_{\FB}$.
\end{proof}
In view of \cref{cor:conc:old}\ref{cor:conc:i},
it is tempting to ask whether we can derive
 a similar conclusion for the equality of $\ran\TFB[]$
 and $\ran \TDR[]$.
 The next example gives a negative answer to this 
 conjecture.  

\begin{example}[$\ran T_{\DR} \neq \ran T_{\FB}$]
\label{ex:rans:neq}
Suppose that $A=\Id$.
Then $T_{\DR}=\tfrac{1}{2}\Id+J_B0$ and $T_{\FB}\equiv J_B0$.
Consequently,
\begin{equation}
X=\ran T_{\DR} \neq \ran T_{\FB}=\stb{J_B0}.
\end{equation}
\end{example}
\begin{proof}
One can easily verify that $J_A=\tfrac{1}{2}\Id$, 
hence $R_A\equiv 0$. 
Therefore, 
\begin{equation}
T_{\DR} =\Id-J_A+J_BR_A
=\Id-\tfrac{1}{2}\Id+J_B0
=\tfrac{1}{2}\Id+J_B0,
\end{equation}
and 
\begin{equation}
T_{\FB} =J_B(\Id-A)=J_B(\Id-\Id)\equiv J_B0,
\end{equation}
and the conclusion readily follows.
\end{proof}
Unlike the Douglas--Rachford operator, where 
we can learn about $\ran \TDR[]$
(see \cite[Corollary~5.3]{BHM15}), we cannot obtain
accurate information about the range of $\TFB[]$
as we show next.
\begin{lem}
\label{lem:ranTFB:inc}
$\ran \TFB[]\subseteq \dom B$.
\end{lem}
\begin{proof}
Indeed, $\ran \TFB[]
\subseteq \ran J_B=\ran(\Id+B)^{-1}=
\dom (\Id+B)=\dom B$.
\end{proof}
The result in \cref{lem:ranTFB:inc},
cannot be improved as we illustrate now.
\begin{example}[$\ran {\TFB[]} \subsetneqq \dom B$]
Suppose 
that $A=\Id$ and that $\dom B$ is not a singleton.
Then \cref{ex:rans:neq} implies that
$ \stb{J_B0}=\ran \TFB[] \subsetneqq \dom B$.
\end{example} 
\begin{example}
[$\ran {\TFB[]} = \dom B$]
Let $C$ be a nonempty closed convex subset of $X$.
Suppose that $A\equiv 0$ and that $B=N_{C}$.
Then 
\cref{eq:res:proj}
implies that
 $\TFB[]=J_B=P_{C}$, hence $\ran \TFB[]=C=\dom B$.
\end{example} 

The \emph{normal problem} (see \cite[Definition~3.7]{Sicon2014})
associated with the ordered pair 
$(A,B)$ is the $v$-perturbed problem
where $v$ is the \emph{minimal displacement
vector} 
defined by
\begin{equation}
\label{def:gap:vec}
v:=v_{\FB}:=P_{\overline{\ran}\bk{\Id-\TFB[]}}0;
\end{equation}
and the corresponding set of 
\emph{normal solutions} is $Z_v$.

\begin{cor}
\label{cor:Zv}
$Z_v=\fix(v+\TFB[])$.
\end{cor}
\begin{proof}
This follows from \cref{prop:eqv}.
\end{proof}
We point out that, even though 
the normal problem is well-defined in
view of 
\cref{fact:inf:dis:vec},
the set of normal solution may or may not be empty,
as we illustrate now.
\begin{example}
[{\bf $Z=\fady$ but normal solutions exist}]
Let $a^*, b^*\in X$ such that $a^*+ b^*\neq 0$.
Suppose that $A\colon X\to X\colon x\mapsto a^*$,
 and that $ B\colon X\to X\colon x\mapsto b^*$.
 Then $Z=\fady$, $\ran(\Id-\TFB[])=\stb{a^*+b^*}$,
 therefore $v=a^*+b^*\in \ran(\Id-\TFB[])$
  and $Z_v=X\neq \fady$.
\end{example}
\begin{proof}
We have $J_B=(\Id+b^*)^{-1}=\Id-b^*$,
and $\TFB[]=J_B(\Id-A)=\Id-(a^*+b^*)$.
Consequently, $\overline{\ran}(\Id-\TFB[])
=\ran(\Id-\TFB[])=\stb{a^*+b^*}$ and 
$v=a^*+b^*\in \ran(\Id-\TFB[])$.
Therefore, $(\forall x\in X)$
$x-\TFB[]x =a^*+b^*=v$, which 
in view of \cref{cor:Zv},
implies that $Z_v=X$, as claimed.
\end{proof}
\begin{example}
[{\bf $Z=\fady$ and normal solutions do not exist}]
Suppose that $X=\RR^2$,
that $U=\menge{(x,y)\in \RR^2}{x>0,y\ge 1/x}$, 
that $V=\RR\times\stb{0}$, 
that $\beta <0$,
that $w=(\beta,0)\neq (0,0)$ 
and that $f = \tfrac{1}{2}d^2_U$.
Set $A=\grad f$ and set $B=w+N_V$. 
Then $\TFB[]=-w+P_VP_U$,
$v=w$, $v\not\in \ran(\Id-\TFB[])$
and therefore $Z_v=\fady$.
\end{example}
\begin{proof}
In view of 
\cref{eq:grad:dist2}
we have 
$A=\Id-P_U$. Moreover \cref{eq:res:proj}
and
\cite[Proposition~23.15(ii)]{BC2011}
implies that $J_B=P_V(\cdot-w)=P_V-w$,
where the last identity uses 
that $P_V$ is linear and that $w\in V$.
Consequently $\TFB[]=J_B(\Id-A)=P_V(\Id-(\Id-P_U))-w
=P_VP_U-w$.
We claim that 
\begin{equation}
\label{e:ran:ex}
{\ran}(\Id-\TFB[])=w+\ran (\Id-P_VP_U).
\end{equation}
Indeed, let $y\in X$.
Then $y\in {\ran}(\Id-\TFB[])$  
$\siff (\exists x\in X)$ such that
$y=w+x-P_VP_Ux
\siff y\in w+\ran(\Id-P_VP_U)$.
It follows from \cref{ex:dis:Map} below
that $\overline{\ran}(\Id-P_VP_U)
=\overline{(\rec U)^\ominus+(\rec V)^\ominus}
=\overline{\RR_{-}^2+V^\perp}
=\RR_{-}^2+(\stb{0}\times \RR)
=\RR_{-}\times \RR
$.
Using \cref{eq:trans:form} applied with $S$
replaced by
$\ran(\Id-P_VP_U)$ 
we have 
$v=w+P_{\overline{\ran}(\Id-P_VP_U)}(-w)=w$.
Consequently \cref{e:ran:ex} becomes 
${\ran}(\Id-\TFB[])=v+\ran (\Id-P_VP_U)$.
Furthermore, 
using \cite[Lemma~2.2(i)]{BB94}
$v\in \ran(\Id-\TFB[])
=v+\ran(\Id-P_VP_U)
\siff 0\in \ran(\Id-P_VP_U)
\siff \fix P_VP_U\neq \fady
\siff U\cap V\neq \fady$,
which does not hold, hence
$Z_v=\fady$ by \cref{prop:eqv}.
\end{proof}

\begin{rem}
Suppose that $A^{-1}$ is firmly nonexpansive.
Then one can define the forward-backward operator 
for the dual pair $(A^{-1},B^{-\ovee})$.
Nonetheless, the self-duality property, which is a key 
feature of $\TDR[]$ (see, e.g., \cite[Corollary~4.3]{JAT2012}
or \cite[Lemma~3.6~on~page~133]{EckThesis}), 
does not hold for $\TFB[]$
as we illustrate in \cref{ex:TFB:not:sd}. 
 \end{rem}
\begin{example}[{\bf $T_{\FB}$ is not self-dual}]
\label{ex:TFB:not:sd}
Suppose that $V$ is a closed linear subspace
of $X$
and let $u\in V\smallsetminus\stb{0}$.
Suppose that $A:X\to X:x\mapsto x-u$
 and that $B=N_V$.
 Then $A^{-1}$ is firmly nonexpansive, 
 however
 \begin{equation}
 u\equiv T_{\FB(A,B)}\neq T_{\FB(A^{-1},B^{-\ovee})} \equiv 0.
 \end{equation}
\end{example}
\begin{proof}
First  note that
$A^{-1}:X\to X:x\mapsto x+u$, hence  
$A^{-1}$ is firmly nonexpansive, as claimed.
Since $B$ is linear
we learn that $B^{-1}$ is linear and
so are $J_{B}$ and $J_{B^{-1}}$
by \cite[Theorem~2.1(xviii)]{BMW2012}.
By
\cite[Proposition~4.1(ii)]{JAT2012}
 and 
\cref{JA:fne:orig}\ref{JA:fne:orig:ii}
we have 
$J_{B^{-\ovee}}=J_{({B^{-1}})^{\ovee}}=
J_{B^{-1}}=\Id-J_B=\Id-P_V=P_{V^{\perp}}
$.
Now,
$T_{\FB(A,B)}=J_B(\Id-A)
=P_V(\Id-\Id+u)=P_V u=u$,
whereas 
$T_{\FB(A^{-1},B^{-\ovee})}
= J_{B^{-\ovee}}(\Id-A^{-1})
=P_{V^\perp}(\Id-\Id-u)
=P_{V^\perp}(-u)
=-P_{V^\perp}(u)\equiv 0$.
\end{proof}

\begin{rem}
Clearly the forward-backward operator is \emph{not} symmetric in $A$
and $B$, however, it is critical to consider the order 
 in \cref{eq:def:FB} when only $A$ is firmly nonexpansive.
 If, in addition, $B$ is firmly nonexpansive 
 we can also define $T_{\FB(B,A)}$.
\end{rem}

\begin{cor}
Suppose that $B:X\to X$ is firmly nonexpansive. Then 
$T_{\FB(B,A)}:=J_A(\Id-B)$ is averaged and 
\begin{equation}
\norm{v_{\FB(A,B)}}=\norm{v_{\FB(B,A)}}.
\end{equation}
\end{cor}
\begin{proof}
Combining 
\cref{cor:conc:old}\ref{cor:conc:iii} 
and \cite[Proposition~3.11]{Sicon2014}
 we have
 $\norm{v_{\FB(A,B)}}
=\norm{v_{\DR(A,B)} }=\norm{v_{\DR(B,A)}}
=\norm{v_{\FB(B,A)}}$.
 \end{proof}

\section{The range of the displacement operator}
\label{s:fb:ran}
Unless otherwise stated, in this section 
we work under the assumption that
\begin{empheq}[box=\mybluebox]{equation*}
\label{T:assmp}
H \text{~~is a finite-dimensional Hilbert space}.
\end{empheq} 
The results in this section provide information on
the range of the displacement map $\Id-\TFB[]$.

\begin{defn}[{\bf nearly convex and nearly equal sets}]
Let $C$ and $D$ be subsets\footnote{Let $C$
be a subset of $H$. We use $\ri C$ to denote the interior 
of $C$ with respect to the affine hull of $C$.
} of $H$.
\begin{enumerate}
\item
We say that $D $ is \emph{nearly convex}\footnote{For 
detailed discussion on the algebra of nearly
convex sets we refer the reader to \cite[Section~3]{Rock70}.}
(see \cite[Theorem~12.41]{Rock98})
if there exists a convex set subset $E$ of
$H$ such that $E\subseteq D\subseteq \overline{E}$.
\item
We say that $C$ and $D$ are \emph{nearly equal}\footnote{For 
detailed discussion on the properties of nearly equal and nearly
convex sets we refer the reader to \cite{BMW2012}.}
if 
\begin{equation}
C\simeq D:\siff \overline{C}=\overline{D}\;\;\text{and}
\;\; \ri C=\ri D.
\end{equation}
\end{enumerate} 
\end{defn}

\begin{fact}
\label{fact:dom:ran:nc}
Let $H$ be a finite-dimensional Hilbert space.
Let $C:H\rras H$ be maximally monotone.
Then
$\dom C$ and $\ran C$ 
are nearly convex.
\end{fact}
\begin{proof}
See
 \cite[Theorem~12.41]{Rock98}.
\end{proof}

\begin{thm}
\label{cor:conc}
Let $H$ be a finite-dimensional Hilbert space.
The following hold:   
\begin{enumerate}
\item
\label{cor:conc:ii}
$\ran (\Id-\TFB[])\simeq\ran A+\ran B $.
\item
\label{cor:conc:ii:i}
Suppose that $A$ and $B$ are affine\footnote{Let $B:X\rras X$.
Then $B $ is an affine relation if $\gra B$ is an affine subspace
of $X\times X$.}.
Then 
$\ran (\Id-\TFB[])
=\overline{\ran} (\Id-\TFB[])
=\ran A+\ran B $.
\end{enumerate}
If, in addition,  $A$ or $B$ is surjective then we additionally
have:
\begin{enumerate}
 \setcounter{enumi}{2}
\item
\label{cor:fd:surj:i}
$\ran (\Id-\TFB[])=X$. 
\item
\label{cor:fd:surj:ii}
$\fix \TFB[]=Z\neq \fady$.
\end{enumerate}
\end{thm}
\begin{proof}
\ref{cor:conc:ii}: 
Note that $A$ is $3^*$ monotone (by
\cref{lem:coc:prop}\ref{lem:coc:3*})
 and $\dom A=X$.
It follows from
\cite[Theorem~5.2]{BHM15} that
$\ran (\Id-\TDR[])\simeq(\dom A-\dom B)\cap(\ran A+\ran B)  $.
Now combine with \cref{cor:conc:old}\ref{cor:conc:i} 
and use that $\dom A=X$. 
\ref{cor:conc:ii:i}:
On the one hand,
 $\ran A$ and
$\ran B$ are closed affine subspaces of $X$,
so is their sum $\ran A+\ran B$.
On the other hand,  
since the resolvent $J_B$ is affine
(see \cite[Theorem~2.1(xix)]{BMW2012}), so are $\TFB[]$
and $\Id-\TFB[]$.
Therefore, in view of 
\ref{cor:conc:ii}, $\ran (\Id-\TFB[])=\overline{\ran} (\Id-\TFB[])
=\overline{\ran A+\ran B }=\ran A+\ran B $.
\ref{cor:fd:surj:i}:
Using \cref{cor:conc}\ref{cor:conc:ii}
we have 
$X=\ri X=\ri (\ran A+\ran B)\subseteq \ran (\Id-\TFB[])
\subseteq \overline{\ran} (\Id-\TFB[])=\overline{ \ran A+\ran B}=X$.
\ref{cor:fd:surj:ii}:
Note that in view of  \cref{FB:collec}\ref{FB:collec:ii}
$0\in \ran (\Id-\TFB[])\siff \fix \TFB[]\neq \fady\siff Z\neq \fady$.
Now combine with \ref{cor:fd:surj:i}.
\end{proof}

In the conclusion of \cref{cor:conc}\ref{cor:conc:ii},
we cannot replace 
near equality by equality
as we illustrate in \cref{ex:neq:not:eq}.
\begin{example}
\label{ex:neq:not:eq}
Suppose that $H=\RR^2$
 and let $f\colon\RR^2\to \left]-\infty,+\infty\right]:
 (\xi_1,\xi_2)\mapsto\max\stb{g(\xi_1),\abs{\xi_2}}$,
 where $g(\xi_1)=1-\sqrt{\xi_1}$ if $\xi\ge 0$ and 
 $g(\xi_1)=+\infty$ otherwise. 
Set\footnote{Let
$f:X\to \left]-\infty,+\infty\right]$ 
be convex, lower semicontinuous, and proper.
We use $f^*$ to denote the \emph{convex conjugate
(a.k.a. Fenchel conjugate)} of $f$, defined by 
$f^*:X\to \left]-\infty,+\infty\right]:x\mapsto 
\sup_{u\in X}(\innp{x,u}-f(x))$.
} 
$A=P_{\RR^2_+}$
and  $B=\pt f^*$.
Then $A$ is firmly nonexpansive 
and B is maximally monotone.
Moreover, $\ran A=\RR^2_+$,
$\ran B= \menge{(\xi_1,\xi_2)}{\xi_1>0, \xi_2\in \RR}
\cup \menge{(0,\xi_2)}{\abs{\xi_2}\ge 1}$,
hence $\ran A+\ran B=\menge{(\xi_1,\xi_2)}{\xi_1\ge 0, \xi_2\in \RR}$
but $\ran (\Id-\TFB[])=
\menge{(\xi_1,\xi_2)}{\xi_1>0, \xi_2\in \RR}
\cup \menge{(0,\xi_2)}{\xi_2\le -1}.$
Therefore
\begin{equation}
\ri(\ran A+\ran B) \subsetneqq \ran (\Id-T)
\subsetneqq \overline{\ran A+\ran B}=\ran A+\ran B.
\end{equation}
\end{example}
\begin{proof}
The claim about firm nonexpansiveness 
of $A$ follows from
e.g., \cite[Equation~1.6~on~page~241]{Zar71:1} 
or \cite[Section~3]{GReich84}
 and maximal monotonicity of
$B$ follows from
\cref{fact:func:sd}\ref{fact:func:sd:mm} applied to $f^*$.
Using \cref{fact:func:sd}\ref{eq:sb:inv:conj}
and \cite[Example~on~page~218]{Rock70} 
we see that 
$\dom \pt f=\ran (\pt f)^{-1}= \ran \pt f^*
=\ran B= \menge{(\xi_1,\xi_2)}{\xi_1>0, \xi_2\in \RR}
\cup \menge{(0,\xi_2)}{\abs{\xi_2}\ge 1}$.
Note that in view of \cref{cor:conc}\ref{cor:conc:ii}
we have $
\menge{(\xi_1,\xi_2)}{\xi_1>0, \xi_2\in \RR}=\ri (\ran A+\ran B)
\subseteq \ran (\Id-T)\subseteq 
 \overline{\ran A+\ran B}
 =\menge{(\xi_1,\xi_2)}{\xi_1\ge0, \xi_2\in \RR}$.
 Therefore we only need to check 
 the points in $\menge{(0,\beta)}{\beta\in \RR}$.
To proceed further we recall that
 (see \cite[Example~6.5]{MPA2016}) 
\begin{equation}
\label{eq:sub:cases}
\pt f(\xi_1,\xi_2)=
\begin{cases}
\fady, &\text{if $\xi_{1}<0$;}\\
\fady,   & \text{if } \xi_1 =0 \text { and }\abs{\xi_2}<1;\\
\RR_{-}\times\stb{1},    &  \text{if }  \xi_1=0 \text { and } \xi_2\ge 1; \\
\RR_{-}\times\stb{-1},   &  \text{if }  \xi_1=0  \text{ and }\xi_2\le -1; \\
\conv\stb{(-\tfrac{1}{2}{\xi_1}^{-{1}/{2}},0),(0,1)},
&  \text{if } {\xi_2}= 1-\sqrt{\xi_1} \text { and } 0<\xi_1<1;\\
\conv\stb{(-\tfrac{1}{2}{\xi_1}^{-{1}/{2}},0),(0,-1)},
&  \text{if } {-\xi_2}= 1-\sqrt{\xi_1} \text { and } 0<\xi_1<1;\\
(-\tfrac{1}{2}{\xi_1}^{-\tfrac{1}{2}},0), &  \text{if } 0<\xi_1<1
 \text { and }1-\sqrt{\xi_1}>\abs{\xi_2};\\
(0,1), &  \text{if $0<\xi_{1}<1$ } \text{ and }  \xi_2> 1-\sqrt{\xi_1};\\
(0,-1), &  \text{if $0<\xi_{1}<1$ } \text{ and } -\xi_2 > 1-\sqrt{\xi_1};\\
\conv\stb{(-\tfrac{1}{2},0),(0,1),(0,-1)} ,
&  \text{if }  \xi_1=1 \text { and } \xi_2=0; \\
\conv\stb{(0,1),(0,-1)},
&   \text {if }\xi_1>1 \text{ and } \xi_2=0;\\
(0,1), &  \text{if $\xi_{1}>1$ and }   \xi_2>0;\\
(0,-1), &  \text{if $\xi_{1}>1$ and }  -\xi_2 > 0.
\end{cases}
\end{equation}
 Let $\beta\in \RR$.
 In view of \cref{prop:eqv}  and \cref{fact:func:sd}\ref{eq:sb:inv:conj} 
 we have 
 \begin{subequations}
 \label{ex:strict:ne}
 \begin{align}
 (0,\beta)\in \ran(\Id-\TFB[])
 &\siff (\exists (\xi_1,\xi_2)\in \RR^2) 
~
 (0,\beta)\in P_{\RR^2_+}  (\xi_1,\xi_2)
 +
 \pt f^* (\xi_1,\xi_2-\beta)\\
 &\quad
 =P_{\RR^2_+}  (\xi_1,\xi_2)+
 (\pt f)^{-1} (\xi_1,\xi_2-\beta)\\
 &\siff  (\exists (\xi_1,\xi_2)\in \RR^2) 
~
 (0,\beta)-P_{\RR_+^2}  (\xi_1,\xi_2)
 \in  (\pt f)^{-1} (\xi_1,\xi_2-\beta)\\
 &\siff (\exists (\xi_1,\xi_2)\in \RR^2)
~(\xi_1,\xi_2-\beta)\in \pt 
f \bk{ (0,\beta)-P_{\RR^2_+}  (\xi_1,\xi_2)}.
\end{align}
\end{subequations}

We argue by cases using \cref{eq:sub:cases}
 and \cref{ex:strict:ne}. 

\emph{Case~1:} 
$\xi_1\ge 0$ and  $\xi_2\ge 0$.
Then $(0,\beta)\in \ran(\Id-\TFB[])$
$\siff  (\exists (\xi_1,\xi_2)\in \RR^2)~
(\xi_1,\xi_2-\beta)\in \pt 
f ( (0,\beta)-P_{\RR_+^2}  (\xi_1,\xi_2))
=\pt f (-\xi_1,\beta-\xi_2))$
$\siff $ [$(\exists (\xi_1,\xi_2)\in \RR^2)~\xi_1=0, 
\xi_2-\beta =1 \text{~ and~} \beta -\xi_2\ge 1$
or $\xi_1=0, \xi_2-\beta =-1 \text{~ and~} \beta- \xi_2\le -1$], 
which is impossible.

\emph{Case~2:} 
$\xi_1\le 0$ and  $\xi_2\le 0$.
Then $(0,\beta)\in \ran(\Id-\TFB[])$
$\siff(\exists (\xi_1,\xi_2)\in \RR^2)~
 (\xi_1,\xi_2-\beta)\in \pt 
f ( (0,\beta)-P_{\RR_+^2}  (\xi_1,\xi_2))
=\pt f (0,\beta))$
$\siff$ 
[$(\exists (\xi_1,\xi_2)\in \RR^2)~
\xi_1\le 0, \xi_2-\beta =1 \text{~ and~} \beta \ge 1$
or $\xi_1\le 0, \xi_2-\beta =-1 \text{~ and~} \beta\le -1$ ]
$\siff$ 
[ $(\exists (\xi_1,\xi_2)\in \RR^2)~  \xi_1\le 0, \xi_2=\beta +1\ge 2 $
or $\xi_1\le 0, \xi_2=\beta -1\le -2$].
Since $\xi_2\le 0$ we conclude that
$\beta\le -1$.

\emph{Case~3:}
$\xi_1> 0$ and  $\xi_2 < 0$.
Then $(0,\beta)\in \ran(\Id-\TFB[])$
$\siff (\exists (\xi_1,\xi_2)\in \RR^2)~
(\xi_1,\xi_2-\beta)\in \pt 
f ( (0,\beta)-P_{\RR_+^2}  (\xi_1,\xi_2))
=\pt f (-\xi_1,\beta)$
$\RA$ [$ \xi_1>0$
and by \cref{eq:sub:cases} $-\xi_1>0$] 
which is impossible.

\emph{Case~4:} $\xi_1< 0$ and  $\xi_2>0$.
Then $(0,\beta)\in\ran(\Id-\TFB[])$
$\siff (\exists (\xi_1,\xi_2)\in \RR^2)~
(\xi_1,\xi_2-\beta)\in \pt 
f ( (0,\beta)-P_{\RR_+^2}  (\xi_1,\xi_2))
=\pt f (0,\beta-\xi_2)$
$\siff $[$\xi_1<0, \xi_2-\beta=1 \text{~and~}\beta-\xi_2\ge 1
\text{~or~} \xi_1<0, \xi_2-\beta=-1 \text{~and~}\beta-\xi_2\le-1$], 
which never occurs.

 Altogether we conclude that
$\ran (\Id-\TFB[])=
\menge{(\xi_1,\xi_2)}{\xi_1>0, \xi_2\in \RR}
\cup \menge{(0,\xi_2)}{\xi_2\le -1}$,
as claimed.
\end{proof}
Suppose that $C$ and $D$
are nonempty nearly convex subsets
of $H$. Then \cite[Proposition~2.12]{BMW2012}  
implies that
\begin{equation}
\label{BMW:cl:nc}
C\simeq D\siff \overline{C}=\overline{D}.
\end{equation}
\begin{lem}
\label{lem:ne:f:f*}
Let $H$ be a finite-dimensional Hilbert space.
Suppose that 
$f\colon H\to \left]-\infty,+\infty\right]$
is convex, lower semicontinuous, and proper.
Then the $\dom \pt f\simeq \dom f$
and $
\ran \pt f\simeq \dom f^*$.
\end{lem}
\begin{proof}
It follows from \cref{fact:dom:ran:nc}
 and \cref{fact:func:sd}\ref{fact:func:sd:mm}
that $\dom \pt f$ is nearly convex. Moreover, 
\cite[Corollary~16.29]{BC2011} implies that
$\overline{\dom}~ \pt f=\overline{\dom} f$.
Therefore \cref{BMW:cl:nc}
implies that $\dom {\pt f}\simeq \dom f$.
Using \cref{fact:func:sd}\ref{eq:sb:inv:conj} we have 
$\ran \pt f=\dom (\pt f)^{-1}=\dom \pt f^*$.
Now apply the same argument to $f^*$.
\end{proof}
We recall that (see \cite[Theorem~3.1]{Zar71:1})
for a nonempty closed convex subset
$C$ of $X$ the following holds\footnote{Let
$C$ be a nonempty closed convex subset of $X$.
The \emph{recession cone} of $C$ is
$\rec C:=\{x\in X~|~x+C\subseteq C\}$,
and the \emph{polar cone} of $C$
is $\textcolor{black}{C^{\ominus}:=\menge{u\in X}{\sup_{c\in
C}\innp{c,u}\le 0}, }
$}:
\begin{equation}
\label{eq:Zar:ran}
\overline{\ran} (\Id-P_C)=(\rec C)^\ominus.
\end{equation} 
\begin{example}
\label{exam:ne:f:f*}
Let $H$ be a finite-dimensional Hilbert space.
Suppose that $C$ is a nonempty closed convex subset of $H$.
Set $f=\iota_C$ and suppose that $A=\pt f=N_C$.
Then $\dom A=C$ and $\ran A\simeq(\rec C)^{\ominus}$.
\end{example}
\begin{proof}
Clearly 
$\dom A=C$.
It follows from \cite[Proposition~23.2(i)]{BC2011},
\cref{JA:fne:orig}\ref{JA:fne:orig:ii}
 and \cref{eq:res:proj}
that $\ran A=\dom A^{-1}
=\ran J_{A^{-1}}=\ran (\Id-J_A)=\ran (\Id-P_C)$.
In view of \cref{eq:Zar:ran}
 we have $\overline{\ran} (\Id-P_C)=(\rec C)^\ominus$.
 Note that $J_{A^{-1}}=\Id-P_C$ is maximally monotone 
 by 
\cref{JA:fne:orig}\ref{JA:fne:orig:ii}\&\ref{JA:fne:orig:i},
 therefore \cref{fact:dom:ran:nc} implies that
  $\ran(\Id-P_C)$ is nearly convex.
 Now apply \cref{BMW:cl:nc}. 
\end{proof}
Suppose that $C_1$ and $C_2$
are nearly convex subsets of $H$ and that
$D_1$ and $D_2$ are subsets of $H$ such that
$C_i\simeq D_i$ for every $i\in \stb{1,2}$.
It follows from 
\cite[Theorem~2.14]{BMW2012} that 
\begin{equation}
\label{eq:neq:sum}
C_1+C_2\simeq D_1+D_2.
\end{equation}

\begin{prop}
\label{prop:func:fd}
Let $H$ be a finite-dimensional Hilbert space.
Suppose
that
$f\colon H\to \RR$ is convex and differentiable 
such that $\grad f$ is nonexpansive 
and that
$g\colon H\to  \left]-\infty,+\infty\right]$
is convex, lower semicontinuous, and proper.
Suppose that $A=\grad f$ and that
$B=\pat g$. 
Then the following hold:
\begin{enumerate}
 \item
\label{prop:func:iii}
 $\ran(\Id -\TFB[])\simeq \dom f^*+\dom g^*$. 
 \end{enumerate}   
 If in addition, $g=\iota_V$ where $V$
 is a nonempty closed convex subset of $H$,
 then we have:
 \begin{enumerate}
 \setcounter{enumi}{1}
 \item  
 \label{prop:func:v}
 $\ran(\Id -\TFB[])\simeq \dom f^*+(\rec V)^\ominus$.
 \end{enumerate}
\end{prop}
\begin{proof}
It follows from \cref{f:grad:fne} that 
$\grad f$ is firmly nonexpansive.
\ref{prop:func:iii}:
Combine \cref{cor:conc}\ref{cor:conc:ii},
\cref{lem:ne:f:f*}
 and 
 \cref{eq:neq:sum}.
 \ref{prop:func:v}:
It follows from \cref{lem:ne:f:f*}
and \cref{exam:ne:f:f*}
respectively 
that $\ran A\simeq \dom f^*$
and $\ran B\simeq (\rec V)^\ominus$.
Now combine with \cref{cor:conc}\ref{cor:conc:ii}
 and 
 \cref{eq:neq:sum}.
\end{proof}
\begin{example}[{\bf range of the displacement 
map of alternating projections}]
\label{ex:dis:Map}
Let $H$ be a finite-dimensional Hilbert space.
Suppose that $U$ and $V$ are nonempty closed convex subsets of 
$X$, that $f=\tfrac{1}{2}d^2_U$ and that $g=\iota_V$.
Suppose that $A=\grad f=\Id-P_U$ and that $B=\pt g=N_V$.
Then 
\begin{equation}
\ran(\Id-\TFB[])=\ran (\Id-P_VP_U)\simeq (\rec U)^\ominus+(\rec V)^\ominus.
\end{equation}
\end{example}
\begin{proof}
It follows from \cref{eq:Zar:ran}
 and \cref{BMW:cl:nc} that 
 $\ran A=\ran (\Id-P_U)\simeq (\rec U)^\ominus$.
On the other hand \cref{exam:ne:f:f*} implies that
$\ran B\simeq (\rec V)^\ominus$.
Now combine with \cite[Theorem~2.12]{BMW2012}.
\end{proof}

\section{Affine operators and applications}
\label{s:aff:app}
\begin{fact}
\label{fact:3seq:eq}
Let $L\colon X\to X$ be linear and nonexpansive,
let $b\in X$ and 
suppose that $T\colon X\to X \colon x\mapsto Lx+b$.
Let $v_T:=P_{ \overline{\ran}(\Id -T)}0$
 and let $x\in X$.
Then 
\begin{equation}
(\forall \nnn)\quad
T^n x+nv_T=(T_{-v_T})^nx
=(v_T+T)^nx.
\end{equation}
\end{fact}
\begin{proof}
See \cite[Theorem~3.2(iv)~and~(v)]{BM2015:AFF}.
\end{proof}

\begin{lem}
Let $L \colon X\to X$ be linear and nonexpansive,
let $b\in X$, suppose that $T\colon X\to X \colon x\mapsto Lx+b$
 and that $v_T:=P_{ \overline{\ran}(\Id -T)}0\in \ran(\Id -T)$.
 Let $x\in X$.
 Then there exists a point $a\in X$
 such that $v_T+b=a-La$ and 
 $v_T+Tx=a+L(x-a)$.
 Moreover  we have
 \begin{equation}
 \label{eq:TL:it}
(\forall \nnn)
\quad
T^n x+nv_T=(T_{-v_T})^nx
=(v_T+T)^nx=a+L^n(x-a)
 \end{equation}
 and 
 \begin{equation}
 \label{eq:fix:vs:fix}
 \fix (v_T+T)=a+\fix L.
 \end{equation}
\end{lem}
\begin{proof}
Note that $v_T\in \ran(\Id -T)
= \ran(\Id -L)-b\siff v_T+b\in \ran(\Id -L)$.
Now let $a\in X$ be such that $v_T+b=a-La$.
The first two identities in \cref{eq:TL:it}
follow from \cref{fact:3seq:eq}.
We prove the last identity in \cref{eq:TL:it}
 by induction. The case $n=0$ is obvious.
 Now suppose that for some $n\in \NN$
 $(v_T+T)^nx=a+L^n(x-a)$.
 Then $(v_T+T)^{n+1}x=v_T+b+L(a+L^n(x-a))
 =v_T+b+La+L^{n+1}(x-a)=a+L^{n+1}(x-a)$.    
 We now turn to 
 \cref{eq:fix:vs:fix}.     
 In view of \cref{eq:TL:it} applied with $n=1$
we have 
$x\in \fix (v_T+T)\siff
x=v_T+Tx\siff x=a+L(x-a)\siff x-a\in \fix L\siff
x\in a+\fix L$, hence $\fix (v_T+T) =a+\fix L$.              
 \end{proof}
\begin{prop}
\label{prop:LT:as:reg}
Let $L \colon X\to X$ be linear and nonexpansive,
let $b\in X$, suppose that $T \colon X\to X \colon x\mapsto Lx+b$
 and that $v_T:=P_{ \overline{\ran}(\Id -T)}0\in \ran(\Id -T)$. 
 Let $x\in X$.
 Then $\fix(v_T+T)\neq\fady$.
 Moreover the following are equivalent:
 \begin{enumerate}
 \item
 \label{prop:LT:as:reg:i}
  $L$ is asymptotically regular.
 \item
  \label{prop:LT:as:reg:ii}
 $L^n x \to P_{\fix L}x$.
 \item
  \label{prop:LT:as:reg:iii}
  $T^n x+nv_T=(v_T+T)^n x=(T_{-v_T})^n x\to P_{\fix(v_T+T)}x$.
 \item
  \label{prop:LT:as:reg:iv}
   $T_{-v_T}=v_T+T$ is asymptotically regular.
 \item
  \label{prop:LT:as:reg:v}
 $(T^n x+nv_T)_\nnn$ is asymptotically regular.
 \end{enumerate}
\end{prop}
\begin{proof}
The proof uses the same techniques as in \cite{BLM16}.
``\ref{prop:LT:as:reg:i}$\siff$\ref{prop:LT:as:reg:ii}":
See 
\cite[Proposition~4]{Baillon76},
 \cite[Theorem~1.1]{Ba-Br-Reich78}, 
 \cite[Theorem~2.2]{BDHP03}
or \cite[Proposition~5.27]{BC2011}.
``\ref{prop:LT:as:reg:ii}$\RA$\ref{prop:LT:as:reg:iii}":
Using \cref{eq:TL:it} and \cref{eq:trans:form}
we learn that
\begin{subequations}
\label{eq:lim:aff}
\begin{align}
T^n x+nv_T&=(T_{-v_T})^nx
=(v_T+T)^nx=a+L^n(x-a)\\
&~\to a+P_{\fix L}(x-a)
=P_{a+\fix L}x=P_{\fix(v_T+T)}x.
\end{align}
\end{subequations}
Now combine with \cref{eq:fix:vs:fix}. 
``\ref{prop:LT:as:reg:iii}$\RA$\ref{prop:LT:as:reg:iv}":
Clear.
``\ref{prop:LT:as:reg:iv}$\RA$\ref{prop:LT:as:reg:v}":
This follows from \cref{fact:3seq:eq}.
``\ref{prop:LT:as:reg:v}$\RA$\ref{prop:LT:as:reg:i}":
Using \cref{eq:TL:it} we have 
$L^n x-
L^{n+1}x
=T^n (x+a)+nv_T-(T^{n+1}(x+a)+(n+1)v_T)
\to 0$.
\end{proof}

Let $\mathcal{B}(X)$ denote the set 
of bounded linear operators on $X$.
We have the following result.
\begin{prop}
\label{prop:rates}
Let $L \colon X\to X$ be linear and nonexpansive,
let $b\in X$, suppose that $T \colon X\to X \colon x\mapsto Lx+b$
 and that $v_T:=
 P_{ \overline{\ran}(\Id -T)}0\in \ran(\Id -T)$. Let $x\in X$
 and let $\mu\in \left]0,1\right[$.
Then the following are equivalent:
\begin{enumerate}
\item
\label{prop:rates:i}
$T^n x+nv_T=(v_T+T)^n x=(T_{-v_T})^nx \to P_{\fix (v_T+T)}x$
$\mu$-linearly.
\item
\label{prop:rates:ii}
$L^n x\to P_{\fix L}x$
$\mu$-linearly.
\item
\label{prop:rates:iii}
$L^n \to P_{\fix L}$
$\mu$-linearly (in $\mathcal{B}(X)$).
\end{enumerate}
\end{prop}
\begin{proof}
Note that $L$ is asymptotically regular 
by \cref{f:av:asreg}.
``\ref{prop:rates:i}$\siff$\ref{prop:rates:ii}":
In view of \cref{eq:TL:it},
\cref{eq:fix:vs:fix}
 and \cref{eq:trans:form}
 we learn that
$T^n x+nv_T-P_{\fix (v_T+T)}x
=(v_T+T)^n x-P_{\fix (v_T+T)}x
=(T_{-v_T})^nx-P_{\fix (v_T+T)}x
=a+L^n(x-a)-P_{a+\fix L}x
=a+L^n(x-a)-a-P_{\fix L}(x-a)
=L^n(x-a)-P_{\fix L}(x-a)
$.
``\ref{prop:rates:ii}$\siff$\ref{prop:rates:iii}":
This follows from \cite[Lemma~2.6]{BLM16}.
\end{proof}

\begin{cor}
\label{cor:lin:rate:fd}
Suppose that $X$ is finite-dimensional.
Let $L \colon X\to X$ be linear, nonexpansive
 and asymptotically regular,
let $b\in X$, set $T \colon X\to X \colon x\mapsto Lx+b$
 and suppose that $v_T:=
 P_{ \overline{\ran}(\Id -T)}0$. 
Let $x\in X$.
Then $v_T\in \ran(\Id -T)$ and 
\begin{equation}
T^n x+nv_T=(v_T+T)^n x=(T_{-v_T})^nx 
\to P_{\fix (v_T+T)}x\quad\text{linearly}.
\end{equation}
\end{cor}
\begin{proof}
Since $X$ is finite-dimensional we learn
that $\ran(\Id -T)$ is a closed affine subspace
of $X$, hence $v_T\in \ran(\Id -T)$.
Now \cref{prop:LT:as:reg} implies that 
$L^n x\to P_{\fix L} x$, which when combined with
\cite[Corollary~2.8]{BLM16} yields 
$L^n x\to P_{\fix L} x$ linearly.
Now apply \cref{prop:rates}
\end{proof}

\begin{thm}[{\bf application to the forward-backward algorithm}]
\label{prop:FB:app}
Suppose that $A$ and $B$ 
are affine and let $x\in X$.
Then the following hold:
\begin{enumerate}
\item
\label{prop:FB:app:i}
$
(\TFB[](\outs[v]{A},\inns[v]{B}))^nx
=(v+T_{\FB})^n x
=(\bk{\TFB[]}_{-v})^n x
=T_{\FB}^n x+nv.
$
\item
\label{prop:FB:app:ii}
If $v\in \ran(\Id-T_{\FB})$ then 
\begin{equation}
(v+T_{\FB})^n x
=(\bk{\TFB[]}_{-v})^n x
=T_{\FB}^n x+nv\to P_{\fix(v+T)}x=P_{Z_v}x.
\end{equation}
\item
\label{prop:FB:app:iii}
We have the implication
\begin{equation}
v=0\in \ran(\Id-\TFB[])~\RA~T_{\FB}^n x\to P_{\fix T}x=P_{Z}x.
\end{equation}
\end{enumerate}
If, in addition, $X$ is finite-dimensional,
then we also have
\begin{enumerate}
 \setcounter{enumi}{3}
\item
\label{prop:FB:app:iv}
$v\in \ran(\Id-T_{\FB})$ and 
\begin{equation}
(v+T_{\FB})^n x
=(\bk{\TFB[]}_{-v})^n x
=T_{\FB}^n x+nv\to P_{\fix(v+T)}x=P_{Z_v}x
\text{~linearly}.
\end{equation}
\item
\label{prop:FB:app:v}
We have the implication
\begin{equation}
v=0~\RA~T_{\FB}^n x\to P_{\fix T}x=P_{Z}x\text{~linearly}.
\end{equation}
\end{enumerate}
\end{thm}
\begin{proof}
\cref{FB:collec}\ref{FB:collec:i:i}
implies that $\TFB[]$ is asymptotically regular
and, since $J_B$ is affine,
(see \cite[Theorem~2.1(xix)]{BMW2012})
so is $v+\TFB[]$.
\ref{prop:FB:app:i}:
The first identity follows from \cref{eq:TFB:shift}
applied with $w$ replaced by $v$.
Now combine with \cref{fact:3seq:eq}.
\ref{prop:FB:app:ii}:
Combine \cref{prop:LT:as:reg}
 and \cref{cor:Zv}.
\ref{prop:FB:app:iii}:
This is a direct consequence of 
\ref{prop:FB:app:ii}.
\ref{prop:FB:app:iv} \& 
\ref{prop:FB:app:v}: Combine 
\cref{cor:lin:rate:fd} with
\ref{prop:FB:app:ii} and 
\ref{prop:FB:app:iii}, respectively.
\end{proof}

\begin{example}
\label{ex:aff}
Let $L:X\to X$ be linear and firmly nonexpansive,
let $b\in X$ and suppose that $U$ is 
an affine subspace of $X$.
Suppose that  $A:X\to X:x\mapsto Lx+b$
and that $B=N_U$. Then the following hold\footnote{Suppose
that $U$ is a closed affine subspace of $X$.
We use $\parl U$ to denote the parallel space of $U$
defined by $\parl U:=U-U$.}:
\begin{enumerate}
\item
\label{ex:aff:i}
$Z_v=(v+U) \cap (L^{-1}((\parl U)^\perp-b+v))$.
\end{enumerate}
If, in addition, $X$ is finite-dimensional then we also have:
\begin{enumerate}
 \setcounter{enumi}{1}
\item
\label{ex:aff:ii}
$\ran(\Id-\TFB[])=\ran L+(\parl U)^\perp+b$.
\item
\label{ex:aff:iii}
$v=P_{\parl U\cap \ker L}b$.
\end{enumerate}

\end{example}
\begin{proof}
\ref{ex:aff:i}:
Let $x\in X$. Then $x\in Z_v\siff 0\in Lx+b-v+N_U(x-v)=Lx+b-v+(\parl U)^\perp$
 $\siff $ [$x-v\in U$ and $Lx\in (\parl U)^\perp-b+v$] 
 $\siff$ [$ x\in v+U$ and  $Lx\in (\parl U)^\perp-b+v$] 
 $\siff x\in (v+U) \cap (L^{-1}((\parl U)^\perp-b+v))$.
\ref{ex:aff:ii}:  
Using \cref{cor:conc}\ref{cor:conc:ii:i}
 we have
 \begin{subequations}
 \label{eq:ex:aff}
 \begin{align}
 \ran(\Id-\TFB[])&=\overline{\ran }(\Id-\TFB[])
 =\ran A+\ran B\\
 &=\ran L+b+\ran N_U=\ran L+(\parl U)^\perp+b.
 \end{align}
 \end{subequations}
\ref{ex:aff:iii}:
Using \cref{lem:coc:prop}\ref{lem:coc:mm} we learn that $L$ is
(maximally) monotone.
Combining \ref{ex:aff:ii},
\cref{eq:trans:form}, 
\cref{eq:ex:aff}, \cite[Theorem~2.19]{conway} and 
\cite[Proposition~20.17]{BC2011}
we have
   \begin{subequations}
 \begin{align}
 v&=P_{\overline{\ran }(\Id-\TFB[])}0=P_{\ran L+(\parl U)^\perp+b}0
 =b-P_{\ran L+(\parl U)^\perp}b\\
 &=P_{(\ran L+(\parl U)^\perp)^\perp}b
 =P_{(\ran L)^\perp\cap(\parl U)}b=P_{\ker L^*\cap \parl U}b
 =P_{\ker L\cap \parl U}b.
 \end{align}
   \end{subequations}
 \end{proof}

\begin{example}[{\bf MAP in the affine-affine feasibility case}]
\label{ex:MAP:app}
Suppose that $U$ and $V$ 
are closed linear subspaces of $X$.
Let $w\in X$. 
Suppose that
$f=\tfrac{1}{2}d^2_{w+U}$, that $g=\iota_{w+V}$,
that $A=\grad f$
and that $B=\pt g$.
Then $(\forall n\in \NN)$
\begin{equation}
\label{eq:MAP:iter:shft}
(\TFB[])^n
=(P_{w+V}P_{w+U})^{n}
=(P_VP_U)^{n}(\cdot-w)+w.
\end{equation}
\begin{proof}
Indeed, let $x\in X$. 
It follows from \cref{lem:FB:MAP}
applied with $(U,V) $ replaced by
$(w+U,w+V) $ and \cref{eq:trans:form}
that $\TFB[]=P_{w+U}P_{w+V}x=P_{w+V}(P_U(x-w)+w)
=P_V(P_U(x-w)+w-w)+w=P_VP_U(x-w)+w$.
Now \cref{eq:MAP:iter:shft} follows by simple induction.
\end{proof}
\end{example}
We now provide an application of the 
forward-backward algorithm that employs
Pierra's
product space technique introduced in \cite{Pierra84}.
 For a
general and more flexible 
framework of using the forward-backward algorithm
to find a zero of the sum of more than two operators we refer the reader 
to the work by Combettes in \cite[Section~2]{ABC10} and
 \cite[Section~5]{CombVu14}.

\begin{prop}[{\bf application to parallel splitting}]
\label{prop:FB:Plsp}
Suppose that $m\in \stb{2,3,\ldots}$.
For every $i\in\stb{1,2,\ldots,m}$,
let $\alpha_i>0$ and suppose that 
$A_i:X\to X$ are $\alpha_i$-cocoercive.
Set ${\bf\Delta}:=
\stb{(x,\ldots,x)\in X^m~|~x\in X}$,
set $\alpha=\min\menge{\alpha_i}{i\in \stb{1,2, \ldots, m}}$,
set ${\bf A}={\ds \times_{i=1}^m}\alpha A_i$,
set ${\bf B}=N_{{\bf\Delta}}$,
set ${\bf T}=T_{\FB({\bf A}, {\bf B})}$,
 let $j:X\to X^m:x\mapsto(x,x,\ldots,x)$,
  and 
  let $e:X^m\to X:(x_1,x_2,\ldots,x_m)
  \mapsto\tfrac{1}{m}\bk{\sum_{i=1}^m x_i}$.  
 Let ${\bf x}\in X^m$ and suppose that 
 ${\bf v}:=P_{\overline{\ran }(\Id-{\bf T})}0\in \ran (\Id-{\bf T})$.
 Then the following hold:
  \begin{enumerate}
 \item
 \label{prop:FB:Plsp:i}
 ${\bf \Delta}^\perp=\stb{(u_1,\ldots,u_m)
 \in X^m~|~\sum_{i=1}^{m}u_i=0}$.
  \item
  \label{prop:FB:Plsp:ii}
  $
 {\bf Z_v}:=Z_{({\bf _vA},{\bf B_v})}
 =({\bf v}+ {\bf \Delta})\cap(
 {\bf A}^{-1}({\bf v}+{\bf \Delta}^\perp)).
$
\item
\label{prop:FB:Plsp:iii}
${\bf v}=0\siff \zer(\sum_{i=1}^mA_i)\neq \fady$.
\item 
\label{prop:FB:Plsp:iv}
$X$ is finite-dimensional $\RA$
 $\ran(\Id-{\bf T})
 \simeq {\ds{\bf \Delta}^\perp
 + \times_{i=1}^m}\ran A_i$.
  \end{enumerate}
If $(\forall i\in\stb{1,2,\ldots,m})$
$A_i$ is affine, then we additionally have:
 \begin{enumerate}
  \setcounter{enumi}{4}
\item
\label{prop:FB:Plsp:v}
$({\bf v+ T})^n {\bf x}
={\bf (T_{-v})}^n  {\bf x}
={\bf T}^n{\bf x}+n{\bf v}
\to P_{\fix ({\bf v+T})} {\bf x}=P_{ {\bf Z_v}} {\bf x}$.
\item
\label{prop:FB:Plsp:vi}
$X$ is finite-dimensional $\RA$
$({\bf v+ T})^n  {\bf x}\to P_{\fix {\bf T}} {\bf x}=P_{ {\bf Z_v}} $ linearly. 
\item
\label{prop:FB:Plsp:vii}
$X$ is finite-dimensional $\RA$
 $\ran(\Id-{\bf T})
= {\ds{\bf \Delta}^\perp
 + \times_{i=1}^m}\ran A_i$.
\end{enumerate}
\end{prop}
\begin{proof}
Note that $(\forall i\in \stb{1,\ldots, m})$
$A_i$ is $\alpha$-cocoercive
hence ${\bf A}$ is firmly nonexpansive.
\ref{prop:FB:Plsp:i}: This is
\cite[Proposition~25.5(i)]{BC2011}.
\ref{prop:FB:Plsp:ii}:
Let ${\bf z}\in X^m$.
Then ${\bf z}\in  {\bf Z_v}$
$\siff  {\bf v}\in N_{{\bf \Delta}}({\bf z}-{\bf v})+{\bf A} {\bf z}$
$\siff$ [$ {\bf z}-{\bf v}\in {\bf \Delta}$ 
and ${\bf A} {\bf z}-{\bf v}\in {\bf \Delta}^\perp$]
$\siff$ [$ {\bf z}\in{\bf v}+ {\bf \Delta}$
 and ${\bf z}\in{\bf A}^{-1}({\bf v}+{\bf \Delta}^\perp)$]
 $\siff {\bf z}\in({\bf v}+ {\bf \Delta})\cap(
 {\bf A}^{-1}({\bf v}+{\bf \Delta}^\perp))$.
\ref{prop:FB:Plsp:iii}:
It follows from \cref{def:gap:vec},
 \cref{FB:collec}\ref{FB:collec:ii}
 applied to ${\bf A} $ and ${\bf B}$
 and \ref{prop:FB:Plsp:i}
 that ${\bf v}= 0\siff \fix {\bf T}\neq\fady
 \siff (\exists {\bf z}\in X^m)$
 such that $0\in {\bf A}{\bf z}+N_{{\bf \Delta}}{\bf z}
 ={\bf A}{\bf z}+{\bf \Delta}^\perp
 \siff $ [${\bf z}\in {\bf \Delta}$ and ${\bf A}{\bf z}\in{\bf \Delta}^\perp$]
 $\siff$ [$(\exists z\in X)~ {\bf z}=(z,z,\ldots z)$
  and $\sum_{i=1}^mA_iz=0$]
  $\siff z\in \zer(\sum_{i=1}^mA_i)$.
\ref{prop:FB:Plsp:iv}:
Apply \cref{cor:conc}\ref{cor:conc:ii}
to ${\bf A}$ and ${\bf B}$
 and note that $\ran {\bf A}=\times_{i=1}^m\ran A_i$.
\ref{prop:FB:Plsp:v} \& \ref{prop:FB:Plsp:vi}:
Apply \cref{prop:FB:app}\ref{prop:FB:app:ii}
 and \ref{prop:FB:app:iv} respectively
to ${\bf A}$ and ${\bf B}$.
\ref{prop:FB:Plsp:vii}:
Apply \cref{cor:conc}\ref{cor:conc:ii:i}
to ${\bf A}$ and ${\bf B}$.
\end{proof}

\section{Some algorithmic consequences}
\label{s:alg:consq}
In this section we make use of the following useful fact that is well-known in analysis. 
\begin{fact}
\label{Fact:Knopp}
Suppose that $(a_n)_\nnn$ is a decreasing sequence 
of nonnegative real numbers such that 
$\sum_{n=0}^{\infty} a_n<+\infty$. Then 
\begin{equation}
na_n\to 0.
\end{equation}
\end{fact}
\begin{proof}
See \cite[Section~3.3,~Theorem~1]{Knopp56}.
\end{proof}
\begin{lem}
Let $\linop\colon X\to X$ be linear, nonexpansive
and asymptotically regular, 
let $b\in X$, and suppose that
$T\colon X\to X\colon x\mapsto \linop x + b$
and that $v_{T}:=
P_{\overline{\ran}(\Id-T)}\in \ran(\Id-T)$. 
Let $x\in X$. Then the sequence 
$(\norm{T^n x-T^{n+1}x-v_T})_\nnn$
is a decreasing sequence of nonnegative real numbers
that converges to $0$. 
\end{lem}

\begin{proof}
Let $n\in \NN$.
It follows from 
\cref{fact:3seq:eq}
that $T^n x+nv_T=(v_T+T) ^n x$.
Moreover, since $L$ is nonexpansive 
so is
$v_T+T$.
Now
\begin{align}
\norm{T^n x-T^{n+1}x-v}&=\norm{T^n x+nv_T-(T^{n+1}x+(n+1)v_T)}\nonumber\\
&=\norm{(v_T+T)^nx-(v_T+T)^{n+1}x}\nonumber\\
&\le \norm{(v_T+T)^{n-1}x-(v_T+T)^{n}x}\nonumber\\
&=\norm{T^{n-1} x+(n-1)v_T-(T^n x+nv_T)}\nonumber\\
&=\norm{T^{n-1} x-T^n x-v_T}.
\end{align}
The claim about convergence follows from
\cref{prop:LT:as:reg}.
\end{proof}

\begin{thm}
\label{thm:algo:fb:app}
Let $\linop\colon X\to X$ be linear, nonexpansive
 and asymptotically regular, 
let $b\in X$, and suppose that
$T\colon X\to X\colon x\mapsto \linop x + b$
and that $v_T:=P_{\overline{\ran}(\Id-T)}\in \ran(\Id-T)$. 
Let $x\in X$ and set 
\begin{equation}
(\forall \nnn)\quad
x_n:=T^n x+n(T^{n^2} x-T^{n^2+1} x).
\end{equation}
Then $x_n\to P_{\fix (v_T+T)}x$. 
\end{thm}
\begin{proof}
We have 
\begin{align}
\norm{x_n-(v_T+T)^n x}&=\norm{T^n x
+n(T^{n^2} x-T^{n^2+1} x)-(T^n x+nv_T)}
\nonumber\\
&=n\norm{T^{n^2} x-T^{n^2+1} x-v_T}=\sqrt{n^2} 
\norm{T^{n^2} x-T^{n^2+1} x-v_T}\to0,
\end{align} 
where the limit follows by applying \cref{Fact:Knopp}
with $a_n$ replaced by $\normsq{T^nx-T^{n+1}  x-v_T}$.
It follows from \cref{prop:LT:as:reg}
 that $(v_T+T)^n x\to P_{\fix (v+T_T)}x$, hence
the conclusion follows. 
\end{proof}
\begin{cor}
Suppose that $A$ and $B$ are affine
 and that $v\in \ran(\Id-T_{\FB})$.
Let $x\in X$ and set 
\begin{equation}
(\forall \nnn)\quad
x_n:=T_{\FB}^n x+n(T_{\FB}^{n^2} x-T_{\FB}^{n^2+1} x).
\end{equation}
Then $x_n\to P_{\fix (v+T_{\FB})}x=P_{Z_v}x$. 
\end{cor}
\begin{proof}
Combine \cref{FB:collec}\ref{FB:collec:i}, 
\cref{f:av:asreg}, 
\cref{thm:algo:fb:app} 
 and 
 \cref{prop:FB:app}\ref{prop:FB:app:ii}.
\end{proof}
\section*{Acknowledgement}
The author thanks Heinz Bauschke 
for his constructive comments and
support.


{\small 
\begin{thebibliography}{999}
\seppthree
\bibitem{AT}
H.\ Attouch and M.\ Th\'era,
A general duality principle for the sum of two operators,
\emph{Journal of Convex Analysis}~3 
(1996), 1--24.
\bibitem{ABC10}
H.\ Attouch, L. M.\ Brice\~{n}o-Arias and P. L.\ Combettes,
A parallel splitting method for coupled monotone inclusions,
\emph{SIAM Journal on Control and Optimization}~vol. 48 (2010), 
3246--3270.

\bibitem{Ba-Br-Reich78}
J.B.\ Baillon, R.E.\ Bruck and S.\ Reich, 
On the asymptotic behavior of nonexpansive 
mappings and semigroups in Banach spaces, 
\emph{Houston Journal of Mathematics}~4 
(1978), 1--9.


\bibitem{Baillon76}
J.B.\ Baillon,
Quelques propri\'et\'es de convergence asymptotique
pour les contractions impaires,
\emph{Comptes rendus de l'Acad\'emie des Sciences}
238(1976), Aii, A587-A590.

\bibitem{BH1979}
J.-B.\ Baillon and G.\ Haddad, Quelques propri\'et\'es des op\'erateurs 
angle-born\'es et n-cycliquement monotones,
\emph{Israel Journal of Mathematics}~26 (1977), 137--150.

\bibitem{BM2015:AFF}
H.H.\ Bauschke and W.M.\ Moursi,
The Douglas--Rachford algorithm for two 
(not necessarily intersecting) affine subspace,
\emph{SIAM Journal in Optimization}~26,
968--985,
2016.

\bibitem{JAT2012} H.H.\ Bauschke, R.I.\ Bo\c{t}, 
W.L.\ Hare and W.M.\ Moursi,
Attouch--Th\'era 
duality revisited: paramonotonicity and operator splitting, 
\emph{Journal of Approximation Theory}~164 
(2012), 1065--1084. 

\bibitem{BC2011}
H.H.\ Bauschke and P.L.\ Combettes,
\emph{Convex Analysis and Monotone 
Operator Theory in Hilbert Spaces},
Springer, 2011.


\bibitem{BDHP03}
H.H.\ Bauschke, F.\ Deutsch, H.\ Hundal 
and S.-H.\ Park: Accelerating the convergence of 
the method of alternating projections, 
\emph{Transactions of the American 
Mathematical Society}~355 
(2003), 3433--3461.  

\bibitem{BDM:LNA:15}
H.H.\ Bauschke, M.N.\ Dao and W.M.\ Moursi, 
On Fej\'{e}r monotone sequences and nonexpansive mappings,
\emph{Linear and Nonlinear Analysis},
vol.~1,
pp.~287--295,
2015.


\bibitem{BHM15}
H.H.\ Bauschke, W.L.\ Hare and W.M.\ Moursi, 
On the range of the Douglas--Rachford operator, 
\emph{Mathematics of Operations Research}, in press.

\bibitem{Sicon2014}
H.H.\ Bauschke, W.L.\ Hare and W.M.\ Moursi, 
Generalized solutions 
for the sum of two maximally monotone operators, 
\emph{SIAM Journal on 
Control and Optimization}~52 (2014), 
1034--1047.

\bibitem{BB94}
H.H. Bauschke and J.M. Borwein,
Dykstra's alternating projection algorithm 
for two sets, 
\emph{Journal of Approximation 
Theory}~79 (1994), 418--443.

\bibitem{BWY2014}
H.H.\ Bauschke, X.\ Wang and L.\ Yao,
 Rectangularity and paramonotonicity 
 of maximally monotone operators, 
 \emph{Optimization}~63 
 (2014), 487--504.

\bibitem{BMW2012}
H.H.\ Bauschke, S.M.\ Moffat and X.\ Wang,
 Firmly nonexpansive mappings 
 and maximally monotone operators: 
 correspondence and duality, 
 \emph{Set-Valued and Variational Analysis}~20 (2012), 
 131--153.
 
 \bibitem{BLM16}
H.H.\ Bauschke, B.\ Lukens and W.M.\ Moursi, 
Affine nonexpansive operators, Attouch--Th\'{e}ra 
duality and the Douglas--Rachford algorithm,
\href{http://arxiv.org/pdf/1603.09418v1.pdf}{\texttt{arXiv:1603.09418 [math.OC]}}.

 
\bibitem{BorVanBook}
J.M.\ Borwein and J.D.\ Vanderwerff,
\emph{Convex Functions},
Cambridge University Press, 2010.

\bibitem{Br-H} H.\ Brezis and A.\ Haraux, 
Image d'une Somme d'op\'{e}rateurs Monotones et Applications,
\emph{Israel Journal of Mathematics}~23 (1976), 165--186.

\bibitem{Brezis}
H.\ Brezis,
\emph{Operateurs Maximaux Monotones et
Semi-Groupes de Contractions dans 
les Espaces de Hilbert},
North-Holland/Elsevier, 1973. 

\bibitem{Br-Reich77}
R.E.\ Bruck and S.\ Reich, 
Nonexpansive projections and 
resolvents of accretive operators in
Banach spaces, \emph{Houston 
Journal of Mathematics}~3 
(1977), 459--470.
\bibitem{BurIus}
R.S.\ Burachik and A.N.\ Iusem,
\emph{Set-Valued Mappings and Enlargements
of Monotone Operators},
Springer-Verlag, 2008.
\bibitem{Burachik}
R.S.\ Burachik and V.\ Jeyakumar,
\emph{Journal of Convex Analysis}~12,
(2005), 279--290.

\bibitem{Comb04}
P.L.\ Combettes,
Solving monotone inclusions via 
compositions of nonexpansive averaged
operators,
\emph{Optimization}~53 
(2004), 475--504.
\bibitem{CombWaj05}
P. L.\ Combettes and V. R.\ Wajs, 
Signal recovery by proximal forward-backward splitting,
\emph{Multiscale Modeling and Simulation}~4 (2005),
1168--1200.
\bibitem{CombDVu10}
P. L.\ Combettes, \DJ inh D\~{u}ng and B. C.\ V\~{u}, 
Dualization of signal recovery problems, \emph{Set-Valued and Variational Analysis}~18 
(2010), 373--404.
\bibitem{CombVu14}
P. L.\ Combettes and B. C.\ V\~{u}, Variable metric 
forward-backward splitting with applications 
to monotone inclusions in duality, \emph{Optimization}~63 (2014),
1289--1318.
\bibitem{conway}
J. B.\ Conway, 
\emph{A Course in Functional Analysis}, 
 Springer-Verlag, 1990.

\bibitem{EckThesis}
J.\ Eckstein,
\emph{Splitting Methods for Monotone Operators with
Applications to Parallel Optimization},
Ph.D.~thesis, MIT, 1989.
\bibitem{Gossez}
 J.-P.\ Gossez, Op\'{e}rateurs monotones non lin\'{e}aires 
 dans les espaces de Banach non
r\'{e}flexifs, 
\emph{Journal of Mathematical Analysis and Applications}, 
34 (1971), 371--395.

%
\bibitem{Iusem98}
A.N.\ Iusem,
On some properties of paramonotone operators,
\emph{Journal of Convex Analysis}~5 (1998), 269--278.

\bibitem{GReich84}
K.\ Goebel and S.\
 Reich, 
 \emph{Uniform Convexity, Hyperbolic Geometry, and Nonexpansive Mappings}, 
 Marcel Dekker, 1984.

\bibitem{Knopp56}
K.\ Knopp, \emph{Infinite sequences and series},
Dover, New York, 1956.

\bibitem{Lemaire97}
B.\ Lemaire, Which fixed point does the iteration method select? 
\emph{Lecture Notes in Economics and
Mathematical Systems}~452 (1979), 154--167. 

\bibitem{L-M79}
 P.L.\ Lions and B.\ Mercier, Splitting 
 algorithms for the sum of two
nonlinear operators.
\emph{SIAM Journal on Numerical 
Analysis}~16(6) (1979), 964--979. 

\bibitem{Minty2}
G.J.\ Minty,
Monotone (nonlinear) operators in Hilbert space,
\emph{Duke Mathematical Journal}~29 (1962), 341--346. 


\bibitem{MPA2016}
S.M.\ Moffat, W.M.\ Moursi
 and X.\ Wang,
 Nearly convex sets: fine properties and domains or
ranges of subdifferentials of convex functions,
 Mathematical Programming, Series A,
DOI: 10.1007/s10107-016-0980-z.

\bibitem{Moreau65}
J.-J.\ Moreau, Proximit\'{e} et dualit\'{e} 
dans un espace hilbertien, 
\emph{Bulletin de la Soci\'{e}t\'{e} Math\'{e}matique 
de France}~93 (1965), 
273--299.

\bibitem{Pazy} A.\ Pazy, Asymptotic behavior 
of contractions in
Hilbert space, \emph{Israel Journal of 
Mathematics}~9 (1971), 235--240.

\bibitem{Pierra84}
G.\ Pierra, Decomposition through 
formalization in a product space,
\emph{Mathematical Programming}~28 (1984), 96--115.

%
\bibitem{Rock98}
R.T.\ Rockafellar and R.J-B.\ Wets,
\emph{Variational Analysis},
Springer-Verlag, 
corrected 3rd printing, 2009.
\bibitem{Rock76}
R.T.\ Rockafellar,
Monotone operators and the proximal point algorithm, 
\emph{SIAM Journal on Control
and Optimization}~14 (1976), 877--898.

\bibitem{Rock70}
R.T.\ Rockafellar,
\emph{Convex Analysis},
Princeton University Press, Princeton, 1970.
%


\bibitem{Rock1970}
R.T.\ Rockafellar, On the maximal monotonicity of subdifferential mappings, 
\emph{Pacific Journal of Mathematics}~33 (1970), 209--216.


%
\bibitem{Simons1}
S.\ Simons,
\emph{Minimax and Monotonicity},
Springer-Verlag,
1998.
%
\bibitem{Simons2}
S.\ Simons,
\emph{From Hahn-Banach to Monotonicity},
Springer-Verlag,
2008.
%
\bibitem{Svaiter}
B.F.\ Svaiter,
On weak convergence of the Douglas--Rachford method,
\emph{SIAM Journal on Control and Optimization}
49 (2011), 280--287.
\bibitem{Tseng91}
 P.\ Tseng, 
 Applications of a splitting algorithm to 
 decomposition in convex programming and variational
inequalities, \emph{SIAM Journal on Control and Optimization}~29
(1991), 119--138.

\bibitem{Zar71:1}
E.H. Zarantonello, 
Projections on convex sets in 
Hilbert space and spectral theory, in: E.H.
Zarantonello (Ed.), \emph{Contributions to Nonlinear 
Functional Analysis}, Academic Press, New York,
(1971),  237--424.

\bibitem{Zeidler2a}
E.\ Zeidler,
\emph{Nonlinear Functional Analysis and Its Applications II/A:
Linear Monotone Operators},
Springer-Verlag, 1990.

\bibitem{Zeidler2b}
E.\ Zeidler,
\emph{Nonlinear Functional Analysis and Its Applications II/B:
Nonlinear Monotone Operators},
Springer-Verlag, 1990.

\bibitem{Zeidler1}
E.\ Zeidler,
\emph{Nonlinear Functional Analysis and Its Applications I:
Fixed Point Theorems},
Springer-Verlag, 1993.

\end{thebibliography}
}
\end{document}